\documentclass[11pt]{amsart}
\usepackage{lmodern}
\usepackage{amsmath, amsthm, amssymb, amsfonts}
\usepackage[normalem]{ulem}
\usepackage{hyperref}
\usepackage{mathtools}
\usepackage{stmaryrd}

\hypersetup{
    colorlinks=true,
    linkcolor=black,
    filecolor=magenta,      
    urlcolor=cyan,
    citecolor=magenta,
    pdftitle={Overleaf Example},
    pdfpagemode=FullScreen,
    }

\usepackage[
  backend=biber,
  style=alphabetic,
  sorting=nty,
  maxcitenames=50,
  maxnames=50
]{biblatex}

\renewbibmacro{in:}{}
\AtEveryBibitem{%
  \clearfield{issn} 
  \clearfield{doi} 
  \clearfield{isbn}
  \ifentrytype{online}{}{
    \clearfield{url}
  }
}

\DeclareSourcemap{
  \maps[datatype=bibtex]{
    \map{
      \step[typesource=arxiv,       typetarget=misc]
      \step[fieldset=entrysubtype, fieldvalue=arxiv]
    }
  }
}

\renewbibmacro{title}{%
  \ifboolexpr{
    test {\iffieldundef{title}}
    and
    test {\iffieldundef{subtitle}}
  }
    {}
    {\printtext[title]{%
       \printfield[titlecase]{title}%
       \setunit{\subtitlepunct}%
       \printfield{subtitle}}%
     \newunit}%
  \printfield{titleaddon}}

\DeclareFieldFormat[article]{title}{\mkbibemph{#1}}

\DeclareFieldFormat{journaltitle}{\textnormal{#1}}

\usepackage{mathrsfs}
\usepackage{graphicx}

\usepackage{verbatim} 
\usepackage{longtable}

\usepackage{mathtools}

\usepackage{tikz}
\usetikzlibrary{decorations.pathmorphing}
\tikzset{snake it/.style={decorate, decoration=snake}}

\usepackage{caption}

\usepackage{tikz-cd}
\usetikzlibrary{arrows}

\newsavebox{\pullback}
\sbox\pullback{
\begin{tikzpicture}%
\draw (0,0) -- (1ex,0ex);%
\draw (1ex,0ex) -- (1ex,1ex);%
\end{tikzpicture}}

\theoremstyle{plain}
\newtheorem{thm}{Theorem}[section]
\newtheorem{cor}[thm]{Corollary}
\newtheorem{lem}[thm]{Lemma}
\newtheorem{prop}[thm]{Proposition}

\newtheorem{question}[thm]{Question}

\theoremstyle{definition}
\newtheorem{defn}[thm]{Definition}

\theoremstyle{remark}
\newtheorem{rmk}[thm]{Remark}

\newcommand{\BM}{{\mathrm{BM}}}

\DeclareMathOperator{\Br}{Br}

\DeclareMathOperator{\Sing}{Sing}
\DeclareMathOperator{\Gal}{Gal}

\DeclareFontFamily{OT1}{rsfs}{}
\DeclareFontShape{OT1}{rsfs}{n}{it}{<-> rsfs10}{}
\DeclareMathAlphabet{\curly}{OT1}{rsfs}{n}{it}

\newcommand{\Pic}{\mathop{\rm Pic}\nolimits}

\DeclareMathOperator{\NS}{NS}
\DeclareMathOperator{\PGL}{PGL}
\DeclareMathOperator{\Bl}{Bl}

    \DeclareFontFamily{U}{wncy}{}
    \DeclareFontShape{U}{wncy}{m}{n}{<->wncyr10}{}
    \DeclareSymbolFont{mcy}{U}{wncy}{m}{n}
    \DeclareMathSymbol{\Sha}{\mathord}{mcy}{"58}

\makeatletter
\newcommand{\customlabel}[2]{%
   \protected@edef\@currentlabel{#2}\label{#1}%
}
\makeatother

\usepackage{tikz}
\usepackage{lmodern}
\usetikzlibrary{decorations.pathmorphing}

\begin{document}
\title[Brauer groups of smooth loci and torsors over relative Jacobians]{Brauer groups of smooth loci in linear systems and torsors over Jacobians of plane curves}
\date{\today}

\newcommand\blfootnote[1]{%
  \begingroup
  \renewcommand\thefootnote{}\footnote{#1}%
  \addtocounter{footnote}{-1}%
  \endgroup
}

\author[M.\ Hartlieb]{Moritz Hartlieb}
\address{Mathematisches Institut, Universität Bonn, Endenicher Allee 60, 53115
Bonn, Germany}
\email{hartlieb@math.uni-bonn.de}

\author[W.\ Pi]{Weite Pi}
\address{Mathematisches Institut, Universität Bonn, Endenicher Allee 60, 53115
Bonn, Germany}
\email{weitepi@math.uni-bonn.de}


\begin{abstract}
    We study Brauer groups of the smooth loci in linear systems on simply connected smooth projective varieties. Under a suitable ampleness condition, we prove that the Brauer group is at most $\mathbb Z/2 \mathbb Z$. This applies when the underlying variety is the projective plane, a very general K3 surface, or a general cubic fourfold. As an application, we compute the Tate--Shafarevich group parametrizing torsors over the relative Jacobians of universal smooth plane curves. Our approach is via a study of the $2$-nodal locus in the linear system.
\end{abstract}

\maketitle

\setcounter{tocdepth}{1} 

\tableofcontents
\setcounter{section}{-1}

\section{Introduction}

Throughout, algebraic varieties are considered over the complex numbers $\mathbb C$.

\subsection{Overview and main results} Given an ample line bundle $L$ on a smooth projective variety $X$, the discriminant locus $\Delta_L$ parametrizes singular elements in the linear system $|L|$. We denote its complement by
\[
|L|_\mathrm{sm} \coloneq |L| \setminus \Delta_L,
\]
which parametrizes smooth divisors. We shall sometimes also write $B = |L|_\mathrm{sm}$ when there is no confusion.

Assume, for the moment, that $X$ is a surface. Denote by $\mathcal{C} \to |L|$ the universal curve, and $\mathcal{C}_B \to B$ (or by slight abuse of notation $\mathcal{C}/B$) its restriction to the smooth locus. Consider the relative Jacobian
\[
\mathcal{J} \coloneq \mathrm{Pic}^0_{\mathcal{C}/B} \to |L|_\mathrm{sm}
\]
whose fiber over a curve $[C]\in |L|_\mathrm{sm}$ is its Jacobian variety $\mathrm{Pic}^0(C)$. This is an abelian scheme. It is a classical problem to study its torsors over the base $B$, which is classified by the \textit{Tate--Shafarevich group}
\[
\Sha(\mathcal{J}) \coloneq H^1(B, \Pic^0_{\mathcal C / B}).
\]
Here we view $\mathrm{Pic}^0_{\mathcal{C}/B}$ as an étale sheaf of automorphisms of the fiberwise Jacobians.

\smallskip

A closely related invariant is the cohomological Brauer group $\mathrm{Br}(-)\coloneq H^2_{\textrm{ét}}(-,\mathbb G_m)$. Indeed, for the relative Jacobian $\mathcal{J} \to B$, we have exact sequences
\begin{equation}\label{eq: leray intro}
\cdots \to \mathrm{Br}(B) \to F^1 \mathrm{Br}(\mathcal{J}) \coloneq \ker\big(\mathrm{Br}(\mathcal{J}) \to H^0(B, R^2 f_* \mathbb G_m)\big) \to H^1(B, \mathrm{Pic}_{\mathcal{J}/B}) \to \cdots
\end{equation}
arising from the Leray spectral sequence applied to the push-forward of $\mathbb G_m$, and 
\begin{equation}\label{eq: exponential intro}
\cdots \to H^1(B, \mathrm{Pic}^0_{\mathcal{J}/B}) \to H^1(B, \mathrm{Pic}_{\mathcal{J}/B}) \to  H^1(B, \mathrm{NS}_{\mathcal{J}/B}) \to \cdots
\end{equation}
arising from the short exact sequence
$$0 \to \Pic^0_{\mathcal J / B} \to \Pic_{\mathcal J/B} \to \NS_{\mathcal J/B} \to 0$$ of \'etale sheaves.

The interconnections between Brauer groups, twisted sheaves, and Tate--Shafarevich twists for linear systems on K3 surfaces have been closely studied in the literature; see for example \cite{Markman, HM1, Mattei-Meinsma, HM2}. 
One of the primary goals of this article is to describe $\Sha(\mathcal{J})$ as well as its generators in the case of smooth plane curves; see Theorem~\ref{thm: tate-shafarevich}. We note that the quasi-projectivity of the base $B$ presents additional technicalities absent in the aforementioned 
K3 setting. Along the way, we obtain surprisingly effective control on the Brauer group of the base $B = |L|_\mathrm{sm}$. We phrase it in a slightly informal way here, which we view as the first main result of this article:

\begin{thm}\label{thm: br base intro}
    Let $X$ be a simply connected smooth projective variety and $L\in \mathrm{Pic}(X)$ be a ``sufficiently ample'' line bundle. Then $\mathrm{Br}(|L|_\mathrm{sm})$ is at most $\mathbb Z/ 2 \mathbb Z$.
\end{thm}

We refer readers to Theorem~\ref{thm:brauer_base} for a more precise version of Theorem~\ref{thm: br base intro}, with a detailed discussion of what \textit{sufficiently ample} means and when this Brauer group actually vanishes. Our primary cases of interest are linear systems on $\mathbb P^2$, very general K3 surfaces, and general cubic fourfolds. In the case of plane curves, we obtain the following

\begin{thm}\label{thm: brauer P2}
    Let $L = \mathcal{O}_{\mathbb P^2}(d)$ with $d\geq 3$. We have
    \[
    \Br(|L|_\mathrm{sm}) \simeq \begin{cases}0 & \text{if } d \text{ is odd,} \\
    \mathbb Z / 2 \mathbb Z &\text{if } d \text{ is even.}\end{cases}
    \]
\end{thm}

Building on Theorem~\ref{thm: brauer P2}, we establish the following result:

\begin{thm}\label{thm: tate-shafarevich}
    For $d\geq 4$, the Tate--Shafarevich group of the relative Jacobian of the universal smooth plane curve of degree $d$ is generated by $\Pic^1_{\mathcal C / B}$, i.e., we have
    $$\Sha(\mathcal{J}) \simeq \mathbb Z / d \mathbb Z \cdot [\Pic^1_{\mathcal C / B}].$$
    For $d=3$, we have $\Sha(\mathcal{J}) \simeq \mathbb Z/6 \mathbb Z$ instead.
\end{thm}

The proof of Theorem~\ref{thm: tate-shafarevich} (in Sections~\ref{sec: symmetric products}--\ref{sec: appendix cubics}) uses the geometry of relative symmetric products $\mathcal{C}^{(n)}_B$ of universal smooth plane curves. For $n = kd$ with $k\geq d-2$, the Abel--Jacobi map
\[
\mathrm{AJ}: \mathcal{C}_B^{(kd)} \to \mathrm{Pic}^{-kd}_{\mathcal{C}/B} \simeq \mathrm{Pic}^0_{\mathcal{C}/B}
\]
is a fibration in projective spaces, i.e.~ a Brauer--Severi variety over $\mathcal{J} = \mathrm{Pic}^0_{\mathcal{C}/B}$. Its period can be computed via classical geometry. On the other hand, for $n\leq d$, we compute $\mathrm{Br}(\mathcal{C}^{(n)}_B)$ via a parallel argument\footnote{As an indication of this, we note that $\mathcal{C}_B^{(0)} =  |L|_\mathrm{sm}$.} as in Theorem~\ref{thm: brauer P2}, using crucially the fact that the evaluation map $\mathcal{C}^{[n]}\to (\mathbb P^2)^{[n]}$ between (relative) Hilbert schemes of points is a projective bundle. Furthermore, we establish an isomorphism
\[
\mathrm{Br}\Big(\mathcal{C}_B^{(kd)}\Big) \simeq \mathrm{Br}\Big(\mathcal{C}_B^{(d)}\Big)
\]
in Proposition~\ref{prop: brauer isom}. Combining this with the period calculation of $\mathcal{C}_B^{(kd)}$ for $k \geq d-2$ gives $\mathrm{Br}(\mathcal{J})$. Finally, we deduce Theorem~\ref{thm: tate-shafarevich} via a careful study of the exact sequences \eqref{eq: leray intro} and \eqref{eq: exponential intro}. The $d=3$ case is special due to the reducibility of the locus of curves having two distinct nodes, and is treated separately in Section~\ref{sec: appendix cubics}.

\begin{rmk}
    It is a natural problem to find the non-trivial Brauer--Severi variety for even degrees $d$ in Theorem~\ref{thm: brauer P2}, as well as the $2$-torsion twist for the Tate--Shafarevich group of  relative Jacobian of cubic curves; we do not treat these in the present article (see also Remark~\ref{rmk: finding Brauer-Severi}). 
\end{rmk}

Theorem~\ref{thm: br base intro} also applies to K3 surfaces and cubic fourfolds. In Sections~\ref{sec: K3} and~\ref{sec: cubic fourfold}, we prove the following results.

\begin{thm}\label{thm: brauer k3}
    Let $(S, L)$ be a primitively polarized K3 surface of Picard rank one and degree at least $10$. Then
    \[
    \mathrm{Br}(|L|_\mathrm{sm}) = 0.
    \]
\end{thm}

In contrast to this theorem, we give examples of non-vanishing $\mathrm{Br}(|L|_\mathrm{sm})$ for low degree K3 surfaces in Section~\ref{sec: K3}. 

\begin{thm}\label{thm: brauer cubic fourfold}
    Let $X$ be a general smooth cubic fourfold. Then we have
    \[
    \mathrm{Br}(|\mathcal{O}_X(1)|_\mathrm{sm}) = 0.
    \]
\end{thm}

Note that the linear systems $|L|$ in Theorems~\ref{thm: brauer k3} and \ref{thm: brauer cubic fourfold} are the bases of Lagrangian fibrations of compact hyperkähler varieties (of K3$^{[n]}$ and OG10 type, cf.~\cite{LSV, Bottini} for the latter). This leads to the natural question whether the smooth locus of a general Lagrangian fibration has vanishing Brauer group.

\subsection{Related works.} In the same spirit as in this article, Vassiliev \cite{Vassiliev} computed some rational cohomology groups of the space of
smooth hypersurfaces of fixed degree in $\mathbb P^n$; see also \cite{Orsola}. Antieau and Meier \cite{AM} computed the Brauer group of the moduli stack $\mathcal{M}_{1,1}$ of elliptic curves over $\operatorname{Spec} \mathbb Z$. Note that we have
\[
(\mathcal{M}_{1,1})_{\mathbb C} \simeq |\mathcal{O}_{\mathbb {P}^2}(3)|_\mathrm{sm} / \mathrm{PGL}_3(\mathbb C).
\]
More recently, di Lorenzo and Pirisi \cite{DiLorenzo-Pirisi} computed, over general algebraically closed fields, Brauer groups of the stack $\mathcal{X}_d$ of smooth plane curves of degree $d\geq 4$, with applications to the moduli stack $\mathcal{M}_3$ of smooth genus three curves. It would be interesting to explore whether our techniques can be applied to study the Brauer groups of the Deligne--Mumford stacks $\mathcal M_{g, n}$ and the Tate--Shafarevich groups of the universal Jacobian over $\mathcal{M}_g$. 

\subsection{Acknowledgments.} We would like to thank David Zhiyuan Bai, Alessio Bottini, Frank Gounelas, and Daniel Huybrechts for discussions related to the topics of this article. Both authors are supported by the ERC
Synergy Grant HyperK (ID~854361). M.\ H.\ is grateful for the support provided by the International Max Planck Research School on Moduli Spaces at the Max Planck Institute for Mathematics in Bonn.

\section{Brauer groups of the complement of discriminant loci}\label{sec: brauer group}

The goal of this section is to prove Theorem~\ref{thm:brauer_base}, a more precise version of Theorem~\ref{thm: br base intro}. We start by spelling out the ampleness condition in Theorem~\ref{thm: br base intro}. For this, we introduce a stratification on the discriminant locus. Let $\Delta_{\geq k}$ denote the closure in $\Delta \coloneq \Delta_L$ of the locus $k\textsf{-sing}$ parametrizing hypersurfaces in $|L|$ with {exactly} $k$ singular points. We define 
\[
\Delta_k \coloneq \Delta_{\geq k} \setminus \Delta_{\geq k+1},
\]
a family of locally closed subsets of $\Delta$. We further denote by $\Delta_\infty$ the locus of hypersurfaces having infinitely many singular points. The subsets $\{\Delta_k\}_{k=1}^\infty$ then give a stratification of $\Delta$. Note that the locus $k\textsf{-sing}$ are constructible; in particular, it contains a dense open subset of $\Delta_{\geq k}$, and a general element in $\Delta_k$ has exactly $k$ singular points.

\begin{rmk}\label{rmk: mu and delta}
    The relationship between our stratification and the more classical stratification via the Milnor number $\mu$, as well as via the $\delta$-invariant when $X$ is a surface, is encapsulated in the following chain of inclusions:
    \begin{equation}\label{eq: chain of inclusions}
    k \textsf{-nodal} \subset \Delta_k \subset \Delta_{\geq k} \subset \{\delta \geq k\} \subset \{\mu \geq k\},
    \end{equation}
where $k$\textsf{-nodal} denotes the locus parametrizing hypersurfaces with exactly $k$ ordinary nodes. On the other hand, an element in $\Delta_k$ need not have exactly $k$ singularities, nor is the equality $\overline{\Delta_k} = \Delta_{\geq k}$ in general true. For the purpose of this article, nevertheless, the readers may safely think of $\Delta_k$ as the ``$k$-nodal locus'' (see also Remark~\ref{rmk: strategy to check dagger} below).
\end{rmk}

Given a very ample $L$, consider the universal singular locus
     \[
     \Sigma\coloneq \{(D,p)\mid p\in \mathrm{Sing}(D)\} \subset |L|\times X
     \]
     and its second projection $\pi: \Sigma \to X$. The fiber $\pi^{-1}(p)$ at a point $p\in X$ is the projectivization of the kernel of the evaluation map
     \[
     \mathrm{ev}_p: H^0(X, L) \longrightarrow H^0(X, L\otimes \mathcal{O}_X/\mathfrak{m}_p^2).
     \]
     The very ampleness of $L$ guarantees that $\mathrm{ev}_p$ is surjective for all $p$; thus $\mathrm{ker}(\mathrm{ev}_p)$ has constant dimension. It follows that $\Sigma$ is a projective bundle over $X$. In particular it is irreducible, and as a byproduct we deduce from the surjection $\Sigma \to \Delta$ that the discriminant locus $\Delta$ is also irreducible.

     \smallskip

     In what follows, we write $\Sigma_k$ (resp.~$\Sigma_{\geq k}$) to be the pull-back of $\Sigma$ to the corresponding locus $\Delta_k$ (resp.~$\Delta_{\geq k}$).

\begin{defn}\customlabel{defn: ampleness}{$(\dagger)$}
    Let $L$ be a line bundle on a smooth projective variety $X$. We say $L$ satisfies the ``ampleness condition'' \ref{defn: ampleness} if the following holds:
\begin{enumerate}
    \item $L$ is very ample.
    \item $\Delta_1 \subset \Delta$ is dense, and $\operatorname{codim}(\Delta_\infty, \Delta)\geq \dim X + 1$.
    \item $\overline{\Delta_{2}}$ is an irreducible divisor in $\Delta$, and is the only full-dimensional component of $\Delta_{\geq 2}$.

    \item $\overline{\Sigma_2}$ is also irreducible.
\end{enumerate}
\end{defn}

\begin{rmk}
    Assumptions (i) and the first part of (ii) implies that $\Delta$ is a divisor in $|L|$. In fact, if $\operatorname{codim} \Delta >1$, then $\mathrm{Br}(|L|_\mathrm{sm}) = 0$ by purity of the Brauer group \cite{Brauerpurity}. Also, assumption (iv) implies the irreducibility part of (iii), but we spell the latter out for clarity.
    \end{rmk}

    \begin{rmk} \label{rmk: strategy to check dagger}
    When $X$ is a surface (e.g.~$\mathbb P^2$ or K3 as in Theorems~\ref{thm: brauer P2} and \ref{thm: brauer k3}), our approach to verify (ii) and (iii) is to first show that the $k${-nodal} locus (sometimes called the \textit{Severi variety}) is irreducible and dense in the $\{\delta = k\}$ locus (sometimes called the \textit{equigeneric locus}), using known results from literature. For (iv), in the $\mathbb P^2$ case we use the fact that the line bundle $L$ in question is $3$-jet ample, while in the K3 case we utilize a tacnode monodromy argument. See Sections~\ref{sec: P2} and \ref{sec: K3} for details. 
\end{rmk}

\begin{thm}\label{thm:brauer_base}
    Let $X$ be a simply connected smooth projective variety and $L\in \mathrm{Pic}(X)$ be a line bundle satisfying \ref{defn: ampleness}. Denote by $|L|_\mathrm{sm}\subset |L|$ the locus parametrizing smooth divisors. Then 
    \[\mathrm{Br}(|L|_\mathrm{sm}) = \begin{cases}0 & \text{  if \:} [\overline{\Sigma_{2}}] \neq 0 \in H^2(\Sigma, \mathbb Z / 2 \mathbb Z),\\ \mathbb Z / 2\mathbb Z & \text{  otherwise}.\end{cases}\]

\end{thm}

 The proof of Theorem~\ref{thm:brauer_base} proceeds through a sequence of lemmas. For the rest of this section, we assume throughout the assumptions of Theorem~\ref{thm:brauer_base}.

\begin{lem}\label{lem:brauer_bm}
 Let $Z$ be a smooth variety and $D \subset Z$ a divisor with complement $U \coloneq Z \setminus D$. Then there are isomorphisms
 $$(\Br(U)[n])/(\Br(Z)[n]) \simeq \ker(H^{\mathrm{BM}}_{2\dim D-1}(D, \mathbb Z / n \mathbb Z) \to H^3(Z, \mathbb Z / n \mathbb Z)).$$
 for all $n \in \mathbb Z$, where the last map is the Gysin morphism.
 \end{lem}
\begin{proof}
The Kummer sequences for $Z$ and $U$ and the Gysin sequence for the pair $(Z, U)$ fit into the commutative diagram
    $$
    \begin{tikzcd}
0 \arrow[r] & \Pic(Z)\otimes \mathbb Z / n \mathbb Z \arrow[r] \arrow[d] & {H^2(Z, \mathbb Z / n \mathbb Z)} \arrow[d]   \arrow[r]                     & {\Br(Z)[n]} \arrow[r] \arrow[d] & 0 \\
0 \arrow[r] & \Pic(U)\otimes \mathbb Z / n \mathbb Z \arrow[r] & {H^2(U, \mathbb Z / n \mathbb Z)} \arrow[r] \arrow[d]              & {\Br(U)[n]} \arrow[r]           & 0 \\
            &                                               & {H_{2\dim D-1}^{\mathrm{BM}}(D, \mathbb Z /n \mathbb Z)} \arrow[d] &                                 &   \\
            &                                                            & {H^3(Z, \mathbb Z /n \mathbb Z)}                                   &                                 &  
\end{tikzcd}
    $$
    with exact rows and columns. Since the restriction $\Pic(Z) \to \Pic(U)$ is surjective, the snake lemma yields the desired isomorphism.
\end{proof}

\begin{rmk}
    For a compact topological space, Borel--Moore homology agrees with standard homology. Thus we drop the superscript ``BM'' in such cases.
\end{rmk}

Applying Lemma~\ref{lem:brauer_bm} with $Z = |L|$ and $D = \Delta$, and recall that the Brauer group of a projective space is trivial, we obtain
\begin{equation}\label{eq: Br isom to BM}
\mathrm{Br}(|L|_\mathrm{sm})[n] \simeq H_{2\dim \Delta -1}(\Delta, \mathbb Z/ n\mathbb Z).
\end{equation}

We compute the RHS of \eqref{eq: Br isom to BM} next. By the denseness of $\Delta_1$, we have $\Delta_{\geq 1} = \Delta$ and $\Delta \setminus \Delta_1 = \Delta_{\geq 2}$. Consider the fiber diagrams
\[\begin{tikzcd}
	{\Sigma_{\geq 2}} & \Sigma & {\Sigma_1} \\
	{\Delta_{\geq 2}} & \Delta & {\Delta_1.}
	\arrow[hook, from=1-1, to=1-2]
	\arrow[from=1-1, to=2-1]
	\arrow[from=1-2, to=2-2]
	\arrow[hook', from=1-3, to=1-2]
	\arrow["f", from=1-3, to=2-3]
	\arrow[hook, from=2-1, to=2-2]
	\arrow[hook', from=2-3, to=2-2]
\end{tikzcd}\]
Note that $f$ is a proper bijective map, thus a homeomorphism. Apply the Gysin sequence to the top and bottom rows, one obtains a long exact sequence of Borel--Moore homology:
\[
\cdots \to H_i(\Sigma) \oplus H_i(\Delta_{\geq 2}) \to H_i(\Delta) \to H_{i-1}(\Sigma_{\geq 2}) \to H_{i-1} (\Sigma) \oplus H_{i-1}(\Delta_{\geq 2}) \to \cdots,
\]
with $\mathbb Z/n \mathbb Z$-coefficient. We are interested in the case $i=2\dim \Delta -1$.

\begin{lem}\label{lemma: br isom to ker}
    Under the above notations, we have 
    \[
    H_{2\dim\Delta -1}(\Sigma, \mathbb Z/ n \mathbb Z) \oplus H_{2\dim \Delta -1}(\Delta_{\geq 2}, \mathbb Z/ n \mathbb Z)= 0
    \]
    for all $n \in \mathbb Z$. Combined with \eqref{eq: Br isom to BM}, we deduce 
    \begin{equation}\label{eq: Br isom to ker}
    \mathrm{Br}(|L|_\mathrm{sm})[n]\simeq \ker\left(\iota: H_{2\dim \Delta -2}(\Sigma_{\geq 2}) \to H_{2\dim \Delta-2} (\Sigma) \oplus H_{2\dim \Delta-2}(\Delta_{\geq 2}) \right),
    \end{equation}
    where the RHS is taken with $\mathbb Z/ n\mathbb Z$-coefficient.
\end{lem}

\begin{proof}
    Assumption (iii) in \ref{defn: ampleness} implies
    \[
    \dim \Delta_{\geq 2} = \dim \Delta -1,
    \]
    so the latter vanishing holds by dimension reasons. On the other hand, recall that $\Sigma$ is smooth and a projective bundle over $X$,
    so we have
    \[
    H_{2\dim\Delta -1}(\Sigma, \mathbb Z/ n \mathbb Z) \simeq H^1(\Sigma, \mathbb Z/ n \mathbb Z) \simeq H^1(X, \mathbb Z/ n \mathbb Z) = 0,
    \]
    where the last vanishing uses our assumption that $X$ is simply connected. The second claim is immediate from the long exact sequence.
\end{proof}

\begin{proof}[Proof of Theorem~\ref{thm:brauer_base}]
    We compute the RHS of \eqref{eq: Br isom to ker}. Note that
    \[
    \dim \Delta_{\geq 2} = \dim \Sigma_{\geq 2} = \dim \Delta -1.
    \]
    In light of assumption (iii) in \ref{defn: ampleness}, we have
    \[
    H_{2\dim \Delta -2}(\Delta_{\geq 2}, \mathbb Z/n \mathbb Z) \simeq \mathbb Z/ n\mathbb Z \cdot [\overline{\Delta_{2}}].
    \]
    By assumption (iv), the subset $\overline{\Sigma_{2}} \subset \Sigma_{\geq 2}$ is irreducible. Combined with assumptions (ii) and (iii), we have
    \[
    H_{2\dim \Delta -2}(\Sigma_{\geq 2}, \mathbb Z/n \mathbb Z) \simeq \mathbb Z/ n\mathbb Z \cdot [\overline{\Sigma_{2}}].
    \]
    Recall that a general element in $\Delta_2$ has exactly two singular points, so the map $\iota$ sends $[\overline{\Sigma_{2}}]$ to $2\cdot [\overline{\Delta_{2}}]$. It follows immediately that $\ker(\iota) = 0$ if $n$ is odd. When $n=2m$ is even, we have instead (note that $\Sigma$ is smooth)
    \[\ker(\iota) = \begin{cases}0 & \text{  if \:} m\cdot [\overline{\Sigma_{2}}] \neq 0 \in H^2(\Sigma, \mathbb Z / 2m \mathbb Z),\\ \mathbb Z / 2\mathbb Z & \text{  otherwise}.\end{cases}\]
    Furthermore, we have 
    \[
    [\overline{\Sigma_{2}}] \neq 0 \in H^2(\Sigma, \mathbb Z / 2 \mathbb Z) \Longleftrightarrow m\cdot [\overline{\Sigma_{2}}] \neq 0 \in H^2(\Sigma, \mathbb Z / 2m \mathbb Z),
    \]
    the claim thus follows.
    \end{proof}

\section{Plane curves}\label{sec: P2}

The purpose of this section is to apply Theorem~\ref{thm:brauer_base} to the projective plane and prove Theorem~\ref{thm: brauer P2}. For this, we need to effectively determine if $[\overline{\Sigma_{2}}] \neq 0 \in H^2(\Sigma, \mathbb Z/ 2\mathbb Z)$ holds. Note that
\[
H^2(\Sigma, \mathbb Z/ 2\mathbb Z) \simeq \mathbb Z/ 2\mathbb Z \cdot H \oplus H^2(X, \mathbb Z/2\mathbb Z)
\]
via the projective bundle $\Sigma \to X$, where $H$ denotes the relative ample class ---  this is also the hyperplane class of $|L|$ restricted to $\Sigma$.

\smallskip

In practice, when $X=S$ is a {surface}, we shall test the vanishing of $[\overline{\Sigma_{2}}]$ by computing its coefficient in $H$, which we denote by $\deg_{\Sigma/S} \Sigma_{\geq 2}$. In geometric terms, this roughly translates into the following enumerative question:

\begin{question}
    Given two general curves in $\Delta$ with a common node and consider the pencil spanned by them. How many $2$-nodal curves appear in this pencil?
\end{question}

The answer is given by the following

\begin{lem}\label{lemma: relative degree}
    Let $L \in \Pic(S)$ be a line bundle satisfying \ref{defn: ampleness}, with $\dim |L| \geq 5$. Then we have
    $$\deg_{\Sigma/S} \Sigma_{\geq 2} = e(S) - 7 + 3 (L^2) + 2 (K_S \cdot L) \equiv e(S) + (L^2) + 1 \mod 2$$
\end{lem}

\begin{proof}
Pick $p \in S$. By \ref{defn: ampleness} and the assumption that the fiber dimension of $\Sigma \to S$ is at least two, we can pick two members $C_1, C_2$ of the linear system $|L|$, such that the following conditions are satisfied: 
\begin{enumerate}
\item The curves $C_1$ and $C_2$ are non-singular away from $p$ and we have $\delta_p(C_i) = 1$;
\item The intersections of $C_1$ and $C_2$ away from $p$ are transverse;
\item The tangent cones of $C_1$ and $C_2$ at $p$ do not intersect away from $p$;
\item The pencil spanned by $C_1$ and $C_2$ meets $\Delta_{\geq 2}$ along the locus $\Delta_{2}$.
\end{enumerate}
Then, we can blow up $S$ in $p$ and the $(L^2)-4$ intersection points of $C_1$ and $C_2$ away from $p$ to obtain the pencil $\widetilde{S} \to \mathbb P^1$, fibers of which correspond to strict transforms of members of the pencil spanned by $C_1$ and $C_2$. Since singularities with $\delta = 1$ are resolved by a blow-up, the singular fibers of $\widetilde{S} \to \mathbb P^1$ are in one-to-one correspondence with members of the pencil that acquire a singularity away from $p$.
Thus, we have 
$$e(\widetilde{S}) = 2e_{\mathrm{gen}} + \deg_{\Sigma/S} \Sigma_{\geq 2},$$
where $e_{\mathrm{gen}} = 2 - 2(g(L) - 1)$ is the Euler characteristic of a non-singular fiber of $\widetilde{S} \to \mathbb P^1$. By the adjunction formula, we have
\[
2 g(L) - 2 = L\cdot (K_S + L).
\]
On the other hand, the Euler characteristic of the blow-up can be computed as $$e(\widetilde{S}) = e(S) + (L^2) - 4 + 1,$$ and the claim follows.
\end{proof}

\begin{proof}[Proof of Theorem~\ref{thm: brauer P2}] We first verify the ampleness condition \ref{defn: ampleness}. Assumption (i) is automatic. For (ii) and (iii), we first note that the codimension of the reducible (resp.~non-reduced) locus in $|\mathcal{O}(d)|$ is well-known to be $d-1$ (resp.~$2d-1$). Thus for $d\geq 4$, it suffices to argue with reduced and irreducible curves (the $d=3$ case is clear). As explained in Remark~\ref{rmk: strategy to check dagger}, we show that the locus in $|L|$ parametrizing integral curves with exactly $k$ ordinary nodes is \textit{irreducible and dense} in the locus $\{\delta = k\}$,
and has dimension $\dim |L| - k$. Indeed, for $S= \mathbb P^2$ these claims are classical; see for example \cite{AC, Z1, Harris}. Thus all inclusions in the chain \eqref{eq: chain of inclusions} except the last one must be equalities, which establishes (ii) and (iii). Lastly, it is known that $\mathcal{O}(d)$ is $3$-jet ample for $d\geq 3$, which implies that the evaluation map
    \[
    \Sigma_\mathrm{bin} \to (\mathbb P^2 \times \mathbb P^2) \setminus (\textrm{diagonal of } \mathbb {P}^2)
    \]
    from the universal $2$-nodal locus, recording the two distinct nodes, is a projective bundle. This implies $\Sigma_\mathrm{bin}$ and thus also $\overline{\Sigma_2}$ are irreducible, giving (iv).


    Now we apply Theorem~\ref{thm:brauer_base}. For $L = \mathcal{O}_{\mathbb P^2}(d)$, we have
    \[
    \deg_{\Sigma/\mathbb P^2} \Sigma_{\geq 2} \equiv 3 + d^2 + 1 \equiv d \mod 2.
    \]
    The claim for odd degrees $d$ follows immediately from Theorem~\ref{thm:brauer_base}. For even $d$, we still need to determine the component of $[\overline{\Sigma_{2}}]$ in $H^2(\mathbb P^2, \mathbb Z/2 \mathbb Z) \simeq \mathbb Z/ 2\mathbb Z \cdot h$. This is obtained by an intersection-theoretic computation. Recall that the class $\Sigma \in H^6(|L|\times \mathbb P^2)$ of the universal singular locus  is given by $(H+(d-1)h)^3$. Writing $b$ to be the coefficient of $h$ in the class $[\overline{\Sigma_{2}}]$ (before modding out $2$), we have
    \begin{align*}
     3(d-1)^2\deg_{\Sigma/ \mathbb P^2} \Sigma_{\geq 2} +3(d-1)b&= \int_{\mathbb |L| \times \mathbb P^2} (\deg_{\Sigma/\mathbb P^2} \Sigma_{\geq2 }\cdot H+bh)(H+(d-1)h)^3 \cdot H^{\dim|L|-2}  \\
    &= \int_{|L| \times \mathbb P^2} [\overline{\Sigma_{2}}] \cdot H^{\dim|L|-2}\\
    &= \int_{|L|} \mathrm{pr}_{1, *}([\overline{\Sigma_{2}}]) \cdot H^{\dim|L|-2} \\
    &= 2 \int_{|L|} [\overline{\Delta_{2}}] \cdot H^{\dim|L|-2} \\
    &= 3(d-1)(d-2)(3d^2 - 3d - 11),
\end{align*}
where the last equality uses the well-known \textit{Severi degree} of $\overline{\Delta_2}$, cf.~\cite{Roberts, Vainsencher}. When $d$ is even, it follows that
\[
b \equiv (d-2)(3d^2 - 3d - 11) \equiv d \equiv 0 \mod 2.
\]
We thus conclude that $[\overline{\Sigma_{2}}]$ vanishes in $H^2(\Sigma, \mathbb Z/ 2\mathbb Z)$, proving the claim for even degrees.
\end{proof}

\begin{rmk}\label{rmk: finding Brauer-Severi}
    It is a natural question to find the Brauer--Severi variety corresponding to the non-trivial element of $\mathbb Z/2\mathbb Z$ in Theorem~\ref{thm: brauer P2} for even $d$. In the $d=4$ case, Di Lorenzo and Pirisi showed that this is related to the theory of theta characteristics for genus three curves; see \cite[Proposition~4.4,~Corollary~4.18] {DiLorenzo-Pirisi}. The fact that their argument involves explicit geometry of bitangents of smooth quartic curves and lines on smooth cubic surfaces may suggest that efforts are required to answer this problem in general.

    Di Lorenzo and Pirisi also showed in the same paper that $\mathrm{Br}(\mathcal{X}_d) \simeq \mathbb Z/\mathrm{gcd}(6,d)\mathbb Z$ for $d\geq 4$, where $\mathcal{X}_d$ is the moduli stack of smooth degree $d$ plane curves. Since we have an isomorphism $\mathcal{X}_d \simeq |\mathcal{O}_{\mathbb P^2}(d)|_\mathrm{sm}/\mathrm{PGL}(3)$, there is an exact sequence
    \[
    0\to \mathbb Z/\mathrm{per}(\alpha)\mathbb Z \to \mathrm{Br}(\mathcal{X}_d) \to \mathrm{Br}(|\mathcal{O}_{\mathbb P^2}(d)|_\mathrm{sm}),
    \]
    where $\alpha$ is the class induced by the $\PGL(3)$-torsor $|\mathcal{O}_{\mathbb P^2}(d)|_\mathrm{sm}$ over $\mathcal{X}_d$. One can thus deduce from Theorem~\ref{thm: brauer P2} and \cite[Theorem~1.3]{DiLorenzo-Pirisi} that
    \[
    \mathrm{per}(\alpha) = 3 \;\textrm{ if \,} 3\mid d, \quad \! \mathrm{per}(\alpha)= 1 \textrm{ otherwise}
    \]
    and that the map $\Br(\mathcal X_d) \to \Br(|\mathcal O_{\mathbb P^2}(d)|_{\mathrm{sm}})$ is surjective.
\end{rmk}

\section{Symmetric products of universal smooth plane curves} \label{sec: symmetric products}

The aim of this section is to compute the Brauer group of the relative $n$-th symmetric product $\mathcal{C}_B^{(n)}$ of universal smooth plane curves for $n\geq 1$ (and $d\geq 4$, although the same strategy will be applied to the $d=3$ case in Section~\ref{sec: appendix cubics}). This is motivated by the fact that for $k \gg 0$, the Abel--Jacobi map
$$\mathcal C_{B}^{(kd)} \to \Pic^{-kd}_{\mathcal C/ B} \simeq \mathcal J$$
is a Brauer--Severi variety, i.e., a fibration in projective spaces. Denote by $\alpha\in \mathrm{Br}(\mathcal{J})$ the corresponding Brauer class. Then we have a short exact sequence 
$$0 \to \langle \alpha\rangle \to \Br(\mathcal J) \to \Br(\mathcal C^{(kd)}_B) \to 0.$$
In Section \ref{sec: torsors over Jacobian}--\ref{sec: appendix cubics}, this will be used to control the Tate--Shafarevich group $H^1(B, \Pic^0_{\mathcal C/ B}),$ which parametrizes torsors over the Jacobian of the universal smooth plane curve.

In order to compute the Brauer group of $\mathcal{C}_B^{(n)}$, we note that for $n = d$, the natural map
$$\mathcal C^{[d]} \to (\mathbb P^2)^{[d]}$$
is a projective bundle since $\mathcal O_{\mathbb P^2}(d)$ is $d$-very ample, where $\mathcal{C}^{[d]}$ is the relative Hilbert scheme of points over the total base $|\mathcal O_{\mathbb P^2}(d)|$. This enables us to proceed by a similar strategy to compute $\mathrm{Br}(\mathcal{C}_B^{(d)})$ as in Section~\ref{sec: brauer group}. Before spelling out the details of this computation, let us discuss how to deduce from it the desired result for symmetric products with more points.

\smallskip

Fix a line $L \subset \mathbb P^2$. Since any smooth plane curve $C \subset \mathbb P^2$ of degree $d$ does not contain $L$, adding the closed subscheme $C \cap L$ of degree $d$ to an arbitrary closed subscheme of degree $n$ yields a well-defined closed immersion
    \begin{align*}
    i_L \colon \mathcal C_B^{(n)} &\hookrightarrow \mathcal C_B^{(n+d)}\\
    Z & \mapsto Z \cup (C \cap L). 
    \end{align*}
The following relative version of \cite[\S 4]{brauerquot} and \cite[Theorem~1.2]{brauersymm}, see also \cite[Section~3.1]{huybrechts2025periodindexproblemhyperkahlermanifolds} for the case of a relative smooth curve with a section, relates the Brauer groups of the various symmetric powers of $\mathcal C_B$.

\begin{prop}\label{prop: brauer isom} For $n\geq 3$, the closed immersion $i_L$ induces isomorphisms
 $$i_L^* \colon \mathrm{Br}(\mathcal C_B^{(n+d)})[m] \simeq \Br(\mathcal C^{(n )}_B)[m]$$ for all $m \in \mathbb Z$.
\end{prop}

\begin{proof}
By the Kummer sequence
$$0 \to \Pic(\mathcal C_B^{(n)}) \otimes \mathbb Z / m \mathbb Z \to H^2(\mathcal C_{B}^{(n)}, \mathbb Z /m \mathbb Z) \to \Br(\mathcal C^{(n)}_B)[m] \to 0,$$
it is enough to show that $i_L$ induces isomorphisms on Picard groups and second cohomology groups with finite coefficients. This is achieved in the following two lemmas.
\end{proof}

\begin{lem}\label{lem: h2 iso}
   For $n\geq 3$, the closed immersion $i_L$ induces isomorphisms
    $$i_L^* \colon H^2(\mathcal C_B^{(n+d)}, \mathbb Z / m \mathbb Z) \simeq H^2(\mathcal{C}_B^{(n)}, \mathbb Z / m \mathbb Z)$$
    for all $m \in \mathbb Z$.
\end{lem}

\begin{proof}
    Let $f_{n}: \mathcal{C}_B^{(n)} \to B$ and $f_{n+d}: \mathcal{C}_B^{(n+d)}\to B$ be the maps to the base. Consider the induced pull-back map
    \[
    i_L^*: R^k (f_{n+d})_* (\mathbb Z/ m \mathbb Z) \to  R^k (f_{n})_* (\mathbb Z/ m \mathbb Z).
    \]
    We claim that this map is an isomorphism for $k\leq 2$. Indeed, this being a sheaf-theoretic question, it suffices to show that
    \[
    i_{L,C}^*: H^k(C^{(n+d)}, \mathbb Z/m \mathbb Z) \to H^k(C^{(n)}, \mathbb Z/ m \mathbb Z)
    \]
    induced by $i_{L,C}: C^{(n)} \hookrightarrow C^{(n+d)}$ is an isomorphism over every point $[C]\in B$. The map $i_{L,C}$ can be factorized as the composition of a sequence of maps adding a single point in $C \cap L$, which we denote by $\iota_j: C^{(n+j)} \hookrightarrow C^{(n+j+1)}$ with $0\leq j \leq d-1$. By Nakai's criterion, see \cite[Chapter~VII, Proposition~2.2]{ACGH85}, the embedding $\iota_j$ realizes $C^{(n+j)}$ as an ample divisor in $C^{(n+j+1)}$. It follows from Lefschetz hyperplane theorem that the pull-back map $\iota_j^*$ is an isomorphism on $H^k(-, \mathbb Z/m \mathbb Z)$ for $k< n+j$, thus for $k\leq 2$ under our assumption. Our claim hence holds, and the lemma follows from comparing the Leray spectral sequences
    \[
    H^p(B, R^q (f_{n})_* (\mathbb Z/m \mathbb Z)) \Rightarrow H^{p+q}(\mathcal{C}_B^{(n)}, \mathbb Z/ m \mathbb Z)
    \]
    and its counterpart for $\mathcal{C}_B^{(n+d)}$.
\end{proof}

\begin{lem}\label{lem: pic iso}
    The closed immersion $i_L$ induces isomorphisms 
    $$i_L^* \colon \Pic(\mathcal C_B^{(n+d)}) \simeq \Pic(\mathcal C_B^{(n)}).$$
\end{lem}
\begin{proof}
   
    Since the fibers of $\mathcal C_B^{(n)} \to B$ are integral, we obtain the diagram
    $$\begin{tikzcd}
0 \arrow[r] & \Pic(B) \arrow[d, "="] \arrow[r] & {\Pic(\mathcal C_B^{(n+d)})} \arrow[d] \arrow[r] & {\Pic(\mathcal C^{(n+d)}_\eta)} \arrow[d] \arrow[r] & 0 \\
0 \arrow[r] & \Pic(B) \arrow[r]                & {\Pic(\mathcal C_B^{(n)})} \arrow[r]             & {\Pic(\mathcal C^{(n)}_\eta)} \arrow[r]             & 0
\end{tikzcd}$$
    of short exact sequences, where $\eta \in B$ denotes the generic point of $B$. On the other hand, the Hochschild--Serre spectral sequence yields the the diagram
    $$\begin{tikzcd}
0 \arrow[r] & {\Pic(\mathcal C^{(n+d)}_\eta)} \arrow[d] \arrow[r] & {\Pic(\mathcal C_{\overline{\eta}}^{(n+d)})^G} \arrow[d] \arrow[r]& \Br(k(\eta)) \arrow[d, "="] \\
0 \arrow[r] & {\Pic(\mathcal C_\eta^{(n)})} \arrow[r]             & {\Pic(\mathcal C_{\overline{\eta}}^{(n)})^{G}} \arrow[r]          & \Br(k(\eta))               
\end{tikzcd}$$
of exact sequences, where $G \coloneq \Gal(k(\overline{\eta})/k(\eta))$. Since $\mathcal C_{\overline{\eta}}$ is a curve over an algebraically closed field, we have
$$\Pic(\mathcal C^{(n)}_{\overline{\eta}}) \cong \Pic(J(\mathcal C_{\overline{\eta}})) \oplus \mathbb Z,$$
see \cite[Theorem~3]{Collino}. From this description, one deduces that the morphism
$$i_L^* \colon \Pic(\mathcal C^{(n+d)}_{\overline{\eta}})^G \to \Pic(\mathcal C^{(n)}_{\overline{\eta}})^G$$
is an isomorphism. The claim then follows by applying the snake lemma to the two diagrams above.
\end{proof}

This establishes Proposition~\ref{prop: brauer isom} via the Kummer sequence as we already explained. Now we compute the Brauer group of relative symmetric products of curves with $n\leq d $ points.

\begin{prop}\label{prop: brauer of symm}
    For $d\geq 4$ and $n\leq d$, the Brauer group of the $n$-th symmetric product of the universal smooth plane curve of degree $d$ is isomorphic to the Brauer group of the base $|L|_{\mathrm{sm}}$, i.e.,
    $$\Br(\mathcal C_B^{(n)}) \simeq \begin{cases}0 & \text{if } d \text{ is odd,} \\
    \mathbb Z / 2 \mathbb Z &\text{if } d \text{ is even.}\end{cases}$$
More precisely, the map $\mathcal C_B^{(n)} \to B$ induces an isomorphism
$$\Br(B) \simeq \Br(\mathcal C_B^{(n)}).$$
\end{prop}
\begin{proof}
We follow the same strategy as in Section ~\ref{sec: P2}.
Since $\mathcal O_{\mathbb P^2}(d)$ is $d$-very ample, the natural evaluation map
$$\mathcal C^{[n]} \to (\mathbb P^2)^{[n]}$$
is a projective bundle. In particular, the relative Hilbert scheme $\mathcal C^{[n]}$ is smooth and satisfies $$H^3(\mathcal C^{[n]}, \mathbb Z / m \mathbb Z) = 0.$$ We also have $\mathrm{Br}(\mathcal{C}^{[n]}) = 0$ since $(\mathbb P^2)^{[n]}$ is rational. By Lemma \ref{lem:brauer_bm}, we thus have\footnote{We write ``$\mathrm{top}$'' to denote the top homological degree, i.e.~twice the dimension of the space considered.}
$$\Br(\mathcal C_B^{(n)})[m] \simeq H_{\mathrm{top}-1}(\mathcal C^{[n]}_{\Delta}, \mathbb Z / m \mathbb Z).$$
In the following, we write $\mathcal{C}^{[n]}_T \coloneq \mathcal{C}^{[n]}_\Delta \times_\Delta T$ for a map $T \to \Delta$.
Consider the map 
\begin{align*}g \colon \mathcal C^{[n]}_\Sigma &\to (\mathbb P^2)^{[n]} \times \mathbb P^2 \qquad \qquad \quad \\
(Z \subset C,\, p \in \Sing(C)) & \mapsto (Z, p),\qquad \qquad \quad 
\end{align*}
which is a projective bundle over the open subset
$$W \coloneqq \begin{cases}
\{(Z, p) \in (\mathbb P^2)^{[n]} \times \mathbb P^2 \mid p \not\in Z \text{ and } Z \cup p \text{ is not contained in any line}\} & \text{if } n = d,\\
\{(Z, p) \in (\mathbb P^2)^{[n]} \times \mathbb P^2 \mid p \not \in Z \} & \text{otherwise,}
\end{cases} $$
see e.g., {\cite[Lemma~2]{computingsym2011}}.
In particular, $g^{-1}(W) \subset \mathcal C^{[n]}_{\Sigma}$ is smooth. Note that the complement of $W$ in $(\mathbb P^2)^{[n]} \times \mathbb P^2$ has codimension at least two and thus we have
\small
$$H^{\BM}_{\mathrm{top}-1}(g^{-1}(W), \mathbb Z /m \mathbb Z) \simeq H^1(g^{-1}(W), \mathbb Z /m \mathbb Z) \simeq H^1(W, \mathbb Z/m \mathbb Z) \simeq H^1((\mathbb P^2)^{[n]} \times \mathbb P^2, \mathbb Z / m \mathbb Z) = 0.$$
\normalsize
The excision exact sequence for $g^{-1}(W) \subset \mathcal C^{[n]}_\Sigma$ then implies 
$$H_{\mathrm{top}-1}(\mathcal C^{[n]}_\Sigma, \mathbb Z / m \mathbb Z) = 0.$$
Since $\mathcal{C}^{[n]}_\Sigma \to \mathcal{C}^{[n]}_\Delta$ is a homeomorphism over $\Delta_1$, it follows that
$$H_{\mathrm{top}-1}(\mathcal C_{\Delta}^{[n]}, \mathbb Z / m \mathbb Z) \simeq \ker(H_{\mathrm{top}}(\mathcal C^{[n]}_{\Sigma_{\geq 2}}, \mathbb Z / m \mathbb Z) \to H_{\mathrm{top}-2}(\mathcal C^{[n]}_{\Sigma}, \mathbb Z / m \mathbb Z) \oplus H_{\mathrm{top}}(\mathcal C_{\Delta_{\geq 2}}^{[n]}, \mathbb Z /m \mathbb Z)).$$
Furthermore, since the fibers of $\mathcal C \to |\mathcal O_{\mathbb P^2}(d)|$ are plane curves, the map $\mathcal C^{[n]}\to |\mathcal O_{\mathbb P^2}(d)|$ is equidimensional, see \cite{luan2023irreduciblecomponentshilbertscheme}, and thus flat by miracle flatness. Flat pullback for Borel--Moore homology then provides us with the commutative diagram
$$\begin{tikzcd}
H_{\mathrm{top}}(\mathcal C^{[n]}_{\Sigma_{\geq 2}}, \mathbb Z / m \mathbb Z) \arrow[r] & H_{\mathrm{top}-2}(\mathcal C^{[n]}_{\Sigma}, \mathbb Z / m \mathbb Z) \oplus H_{\mathrm{top}}(\mathcal C_{\Delta_{\geq 2}}^{[n]}, \mathbb Z / m \mathbb Z)) \\
H_{\mathrm{top}}(\Sigma_{\geq 2}, \mathbb Z /m \mathbb Z) \arrow[r]\arrow[u] & H_{\mathrm{top}-2}(\Sigma, \mathbb Z /m \mathbb Z) \oplus H_{\mathrm{top}}(\Delta_{\geq 2}, \mathbb Z / m \mathbb Z) \arrow[u].
\end{tikzcd}$$

For $d\geq 4$, a general member of $\Delta_{\geq 2}$ is integral, so the fibers of $\mathcal{C}_{\Delta_{\geq 2}}^{[n]}$ are irreducible. It follows that $\mathcal{C}_{\Sigma_{\geq 2}}^{[n]}$ is also irreducible 
and the left vertical map is an isomorphism. This concludes the proof for even degrees. For odd $d \geq 5$, a computation as in the proof of Theorem \ref{thm: brauer P2} shows that the degree of $\mathcal{C}_{\Sigma_{\geq 2}}^{[n]}$ restricted to a fiber of the map
$$\mathcal C_\Sigma^{[n ]} \to (\mathbb P^2)^{[n]} \times \mathbb P^2$$
over a point in the open subset $W$ is odd. As in the proof of Theorem \ref{thm: brauer P2}, we conclude that the top horizontal map in the above diagram is injective for all $m$ and odd $d \geq 5$, proving the vanishing of Brauer groups.
\end{proof}

\begin{rmk}
    For $d=3$, the locus $\mathcal{C}_{\Sigma_{\geq 2}}^{[n]}$ becomes reducible and the computation requires further analysis. Indeed, we shall see in Section~\ref{sec: appendix cubics} that $\mathrm{Br}(\mathcal{C}_B^{(n)}) \simeq \mathbb Z/2\mathbb Z$ for plane cubic curves (with $1\leq n\leq 3$).
\end{rmk}

The following corollary will be crucially used in the proof of Theorem~\ref{thm: tate-shafarevich main text}.

\begin{cor} \label{cor: brauer of symm}
    For $d\geq 4$ and $k\geq 1$, we have 
    \[
    \Br(\mathcal C_B^{(kd)}) \simeq \begin{cases}0 & \text{if } d \text{ is odd,} \\
    \mathbb Z / 2 \mathbb Z &\text{if } d \text{ is even.}\end{cases}
    \]
\end{cor}

\begin{proof}
    This follows immediately from Propositions~\ref{prop: brauer isom} and \ref{prop: brauer of symm}.
\end{proof}

\section{Torsors over the Jacobian of universal smooth plane curves} \label{sec: torsors over Jacobian}

We study the Tate--Shafarevich groups for relative Jacobians of plane curves in this section. In particular, we prove the $d\geq 4$ cases of Theorem~\ref{thm: tate-shafarevich}. We assume $d\geq 3$ throughout this section.

\begin{lem}\label{lem: desc ns}
The relative Neron--Severi sheaf $\NS_{\mathcal{J}/B}$ is the constant sheaf generated by the theta divisor, i.e.
$$\NS_{\mathcal{J}/B} \simeq \mathbb Z \cdot \Theta$$
\end{lem}
\begin{proof}
Fix a prime $\ell$. Then, the injective morphism $\NS_{\mathcal{J}/B} \to R^2f_* \mathbb Z_\ell$, where $f \colon \mathcal J \to B$ denotes the relative Jacobian of the universal plane curve, induces an inclusion
$$H^0(U, \NS_{\mathcal{J}/B}) \subset H^0(U, R^2f_* \mathbb Z_\ell)$$
for any \'etale $U \to B$.
As the very general plane curve $C \subset \mathbb P^2$ satisfies $\NS(\Pic^0(C)) = \mathbb Z \cdot \Theta$, the claim follows.
\end{proof}

\begin{lem}\label{lem: desc picJ}
We have
$$H^0(B, \Pic_{\mathcal{J}/B}) = \begin{cases}
    \mathbb Z \cdot \Theta & \text{if } d \text{ is odd,} \\
    \mathbb Z \cdot 2 \Theta & \text{if } d \text{ is even}.
\end{cases}$$
\end{lem}
\begin{proof}
    Since there exist no non-trivial global relative line bundles of degree zero on $\mathcal C_B$, we have the exact sequence
    $$0 \to H^0(B, \Pic_{\mathcal{J}/B}) \to H^0(B,\NS_{\mathcal{J}/B}) \xrightarrow{\delta} H^1(B, \Pic^0_{\mathcal C / B}).$$
    Moreover, the morphism $\delta$ sends the theta divisor to the torsor $\Pic^{g-1}_{\mathcal C / B}$. This is trivial if and only if there exists a line bundle of relative degree $g-1$ on $\mathcal C / B$. Since $g - 1 = d(d-3)/2$, this is the case when $d$ is odd, while for even $d$ only $\mathrm{Pic}^{2g-2}_{\mathcal{C}/ B}$ is trivial. The claim follows.
\end{proof}

In particular, the map $\Pic(\mathcal{J}) \to H^0(B, \Pic_{\mathcal{J}/B})$ is surjective. From the Leray spectral sequence, we thus obtain the exact sequence
$$0 \to \Br(B) \to F^1\mathrm{Br}(\mathcal{J}) \to H^1(B, \Pic_{\mathcal{J}/B}) \to 0,$$
where exactness on the right is due to the existence of a section $B \to \mathcal{J}$.

\smallskip

We want to compute $\Br(\mathcal J)$. For this, we proceed by fixing $k$ large enough such that the Abel--Jacobi map
\[
\mathcal C^{(kd)}_B \to\mathcal J,\quad 
Z \mapsto \mathcal I_Z(k)
\]
is a Brauer--Severi variety. Let $\alpha \in \Br(\mathcal J)$ denote the corresponding Brauer class, we obtain the short exact sequence 
\begin{equation*}
    0 \to \langle \alpha \rangle \to \Br(\mathcal J) \to \Br(\mathcal C^{(kd)}_{B}) \to 0.
\end{equation*}
It remains to determine the period of $\alpha$, which is equal to the greatest common divisor of the relative degrees of line bundles on $\mathcal C^{(kd)}_B \to \mathcal J$, see e.g., \cite[Lemma~1.1]{Gounelas-Huybrechts}. Let us first describe the Picard group of $\mathcal C_B^{(kd)}.$
\begin{lem}\label{lem: desc pic}
    For any $k \geq 1$, the natural map
    $$\mathcal C^{(kd)}_B \to B \times (\mathbb P^2)^{[kd]}$$
    induces an isomorphism
    $$\Pic(B \times (\mathbb P^2)^{[kd]}) \simeq \Pic(\mathcal C^{(kd)}_B).$$ 
\end{lem}
\begin{proof}
    Since $\mathcal O_{\mathbb P^2}(d)$ is $d$-very ample, the map
    $$\mathcal C^{[d]} \to (\mathbb P^2)^{[d]}$$
    is a projective bundle, which implies the claim for $k = 1$. Now, assume $k \geq 2$ and proceed by induction. Fixing a line $L \subset \mathbb P^2$, we obtain a commutative diagram
 $$\begin{tikzcd}
 {\Pic(B \times (\mathbb P^2)^{[kd]})} \arrow[r] \arrow[d] & \Pic(\mathcal C_B^{(kd)}) \arrow[d] \\
 {\Pic(B \times (\mathbb P^2)^{[(k-1)d]})}  \arrow[r]                                                  & \Pic(\mathcal C_B^{((k-1)d)}),
 \end{tikzcd}$$
where the vertical maps are given by pullback along the (rational) map sending a subscheme $Z \subset \mathbb P^2$ of length $(k-1)d$ to the subscheme $Z \cup (C \cap L)$. This is well-defined away from the codimension two locus of subschemes with non-reduced structure on the line $L$. The right vertical map is an isomorphism by Lemma~\ref{lem: pic iso}. By induction, the lower horizontal map is an isomorphism. In order to conclude, we now show that the vertical map on the left is an isomorphism. By construction, it is the pullback induced by the rational morphism
\begin{align*}
\varphi_L \colon B \times (\mathbb P^2)^{[(k-1)d]} &\dashrightarrow B \times (\mathbb P^2)^{[kd]} \\
([C], Z) &\mapsto ([C], Z \cup (C \cap L)),
\end{align*}
well-defined whenever $Z$ is reduced in a neighborhood of the line $L$.
Note that we have
$$\Pic(B \times (\mathbb P^2)^{[n]}) \simeq \langle H\rangle \oplus \mathbb Z\cdot h^{[n]} \oplus \mathbb Z \cdot \delta \text{ \, for } n \geq 2,$$
where $H,\, h^{[n]},\, 2\delta$ denote respectively the pullback of the restriction of the ample generator on $|\mathcal O_{\mathbb P^2}(d)|$ to $B$, the symmetrization of the ample generator $h \in \Pic(\mathbb P^2)$, and the class of the exceptional divisor of the Hilbert--Chow map.

Since the map $\varphi_L$ commutes with the projections to $B$, we have $\varphi^*_L H = H$. To compute the pullback of $h^{[kd]}$, fix a general line $L'$. Then $h^{[kd]}$ is the class of the locus of subschemes of $\mathbb P^2$ that intersect $L'$ non-trivially. Thus, $\varphi^*_L h^{[kd]}$ is represented by the locus of pairs $([C], Z) \in B \times (\mathbb P^2)^{[(k-1)d]}$ such that $Z$ intersects $L'$ non-trivially, or that $C$ contains the intersection point $L \cap L'$. It follows that 
\[
\varphi^*_L h^{[kd]} \equiv h^{[(k-1)d]} \mod H.
\]
Similarly, we have that $\varphi^*_L (2\delta)$ is represented by the locus of pairs $([C], Z) \in B \times (\mathbb P^2)^{[(k-1)d]}$ such that $Z$ is non-reduced or $L$ is tangent to $C$. Therefore, we have $\varphi_L^* \delta \equiv \delta \mod H.$ The claim follows.
\end{proof}

\begin{lem}\label{lem: computing period}
    Let $\alpha \in \Br(\mathcal J)$ denote the Brauer class induced by $\mathcal C^{(kd)}_B \to \mathcal J$ for $k\gg 0$. Then we have
    $$\operatorname{per}(\alpha) = \begin{cases}d &\text{ if } d \text{ is odd,} \\d/2 & \text{ if } d \text{ is even}.\end{cases}$$
\end{lem}
\begin{proof}
 Recall that the period is equal to the greatest common divisor of the relative degrees of line bundles on $\mathcal C_B^{(kd)}$. 
 By Lemma~\ref{lem: desc pic}, we have 
 \[
 \Pic(\mathcal C_B^{(kd)}) \simeq \Pic(B) \oplus \mathbb Z h^{(kd)} \oplus\delta,
 \]
    where $2\delta$ is the class of the exceptional divisor and $h^{(kd)}$ is described as follows: fixing a general line $l \subset \mathbb P^2$, the corresponding divisor $h^{(kd)}$ is represented over $[C]\in B$ by the locus of subschemes of $C$ that meet $l \cap C$. Clearly, the relative degree of any class pulled back from $\Pic(B)$ is zero. The fiber of $\mathcal C_B^{(kd)} \to \mathcal{J}$ over a point of the form $(C, \mathcal O_C) \in \mathcal{J}$ is given by $\mathbb P(H^0(C, \mathcal O(k)))$.
    Taking a general pencil of degree $k$ plane curves and a given line $l\subset \mathbb P^2$, there are exactly $d$ elements that meet $l\cap C$. Thus $h^{(kd)}$ is of relative degree $d$.
    
    Let us now compute the relative degree of the exceptional divisor, which corresponds to $2\delta$. In $\mathbb P(H^0(C, \mathcal O(k)))$ this is precisely the discriminant locus, which is classically known to be of degree $d(2k+d-3).$ Thus $\delta$ is of degree $d(2k+d-3)/2$. The claim follows.
\end{proof}

Now we can prove Theorem~\ref{thm: tate-shafarevich} in the $d\geq 4$ case, whose statement we recall below.

\begin{thm}\label{thm: tate-shafarevich main text}
    For $d\geq 4$, the Tate--Shafarevich group of the relative Jacobian of the universal smooth plane curve of degree $d$ is generated by $\Pic^1_{\mathcal C / B}$, i.e.
    $$\Sha(\mathcal{J}) \simeq \mathbb Z / d \mathbb Z \cdot [\Pic^1_{\mathcal C / B}].$$
\end{thm}

\begin{proof}
Since $\Pic^k_{\mathcal C / B}$ is the trivial torsor if and only if there is a line bundle of relative degree $k$ on $\mathcal C/B$, we have $\mathbb Z / d \mathbb Z \subset \Sha(\mathcal{J})$, generated by $\Pic^1_{\mathcal C / B}$. On the other hand, we have the exact sequences
$$0 \to H^0(B, \NS_{\mathcal{J} / B}) / H^0(B, \Pic_{\mathcal{J}/B})\to \Sha(\mathcal{J}) \to H^1(B, \Pic_{\mathcal{J} / B}) \to 0,$$
where exactness on the right stems from the vanishing
$H^1(B, \NS_{\mathcal J / B}) \simeq H^1_{\text{\'et}}(B, \mathbb Z) = 0$ due to Lemma \ref{lem: desc ns} (note that we are working with \'etale cohomology groups),
and
$$0 \to \Br(B) \to F^1 \mathrm{Br}(\mathcal{J}) \to H^1(B, \Pic_{\mathcal{J} / B}) \to 0$$
established earlier.
Combining Corollary~\ref{cor: brauer of symm}, Lemma~\ref{lem: desc ns}, Lemma~\ref{lem: desc picJ} and Lemma~\ref{lem: computing period}, we have the following:
if $d$ is odd, then
\[
H^0(B, \NS_{\mathcal{J} / B}) / H^0(B, \Pic_{\mathcal{J}/B}) = 0, \quad |\Br(\mathcal{J}) / \Br(B)| = d.
\]
If $d$ is even, then
\[
H^0(B, \NS_{\mathcal{J} / B}) / H^0(B, \Pic_{\mathcal{J}/B}) = \mathbb Z / 2 \mathbb Z, \quad |\Br(\mathcal{J}) / \Br(B)| = d/2.
\]
In both cases we conclude that $|\Sha(\mathcal{J})| \leq d$, which forces the claimed isomorphism to hold.
\end{proof}

\section{Computations for plane cubic curves}
\label{sec: appendix cubics}

We compute Tate--Shafarevich group for the relative Jacobian of plane cubic curves in this section. In particular, we complete the proof of Theorem~\ref{thm: tate-shafarevich}.

\begin{prop}\label{prop: appendix}
    Let $\mathcal C_B \to B \subset |\mathcal O_{\mathbb P^2}(3)|$ denote the universal smooth plane cubic. Then
    $$\Br(\mathcal C^{(3)}_B) \simeq \mathbb Z/ 2 \mathbb Z.$$
\end{prop}

We first show that this Brauer group is $2$-torsion by continuing the strategy in the proof of Proposition~\ref{prop: brauer of symm}.
\begin{lem}\label{lem: brauersymm2}
    Let $\mathcal C_B \to B \subset |\mathcal O_{\mathbb P^2}(3)|$ denote the universal smooth plane cubic. Then
    $$2\cdot\Br(\mathcal C^{(3)}_B) = 0.$$
\end{lem}
\begin{proof}
In the $d = 3$ case, the general member of $\Delta_{\geq 2}$ is the union of a smooth conic and a line, intersecting transversally in $2$ points. The relative Hilbert scheme $\mathcal C^{[3]}_{\Sigma_{\geq 2}}$ thus consists of the four components
$$\mathcal C^{[c, l]}_{\Sigma_{\geq 2}} \text{\;\, for \,} c  + l = 3,$$
where $c$ denotes the number of points on the conic and $l$ denotes the number of points on the line (generically). 
Recall that we have 
$$\Br(\mathcal C^{(3)}_B)[n] \simeq \ker(H_{\mathrm{top}}(\mathcal C^{[3]}_{\Sigma_{\geq 2}}, \mathbb Z / n\mathbb Z) \to H_{\mathrm{top}-2}(\mathcal C^{[3]}_{\Sigma}, \mathbb Z / n \mathbb Z) \oplus H_{\mathrm{top}}(\mathcal C_{\Delta_{\geq 2}}^{[3]}, \mathbb Z / n \mathbb Z)).$$
Since the map $\mathcal C^{[3]}_{\Sigma_{\geq 2}} \to \mathcal C^{[3]}_{\Delta_{\geq 2}}$ is generically $2$-to-$1$ on each irreducible component, we have
$$\ker(H_{\mathrm{top}}(\mathcal C^{[3]}_{\Sigma_{\geq 2}}, \mathbb Z / n\mathbb Z) \to H_{\mathrm{top}}(\mathcal C_{\Delta_{\geq 2}}^{[3]}, \mathbb Z / n \mathbb Z)) \simeq (\mathbb Z / \gcd(2, n))^{\oplus 4}$$
and thus $\Br(\mathcal C^{(3)}_B) \subset (\mathbb Z / 2 \mathbb Z)^{\oplus 4}$. The claim follows.
\end{proof}

In order to get a more precise description of $\Br(\mathcal C^{(3)}_B)$ via the above strategy, one would have to analyze the irreducible components of $\mathcal C^{[3]}_{\Sigma_{\geq 2}}$. 
Instead of this, we first compute the Brauer group $\Br(\mathcal C_B)$ of the universal smooth plane cubic. We still follow the strategy in Proposition~\ref{prop: brauer of symm}, but with an explicit resolution of singularities of the universal singular cubic as an extra ingredient.

\begin{lem}\label{lem: appendixc0}
    Let $\mathcal C_B \to B \subset |\mathcal O_{\mathbb P^2}(3)|$ denote the universal smooth plane cubic. Then
    $$\Br(\mathcal C_B) \simeq \mathbb Z/2\mathbb Z.$$
\end{lem}
\begin{proof}

Note that ${\mathcal C}_{\Sigma_{\geq 2}} = \mathcal Q \cup \mathcal L$ splits into two irreducible components, since a general element $C \in \Delta_{\geq 2}$ splits as the union of a conic and line $C = Q \cup L$. We resolve the singularities of $\mathcal{C}_\Sigma$ as follows: Let $\mathbf{D} \subset \mathbb P^2 \times \mathbb P^2$ denote the diagonal, and consider the locus
$$\widetilde{\mathcal C}_{\Delta} \coloneq\overline{\{([C], p, q) \in \Delta \times ((\mathbb P^2 \times \mathbb P^2) \setminus \mathbf D) \mid p \in C,\, q \in \Sing(C)\}} \subset \Delta \times \Bl_{\mathbf{D}}(\mathbb P^2 \times \mathbb P^2).$$
The fiber of $\widetilde{\mathcal C}_\Delta \to \Bl_{\mathbf{D}}(\mathbb P^2 \times \mathbb P^2)$ over a point on the exceptional divisor, corresponding to a point $q \in \mathbb P^2$ and a tangent direction, is the locus of plane cubic curves that are singular at $q$ such that the tangent cone at $q$ contains the prescribed tangent direction. It follows from this description that the morphism
$$\widetilde{\mathcal C}_\Delta \to \Bl_{\mathbf D}(\mathbb P^2 \times \mathbb P^2)$$
is a projective bundle (of relative dimension $5$) and, in particular, $\widetilde{\mathcal C}_\Delta$ is smooth and resolves the universal singular cubic by the natural maps $\widetilde{\mathcal C}_{\Delta} \to \mathcal{C}_\Sigma \to \mathcal{C}_\Delta$, forgetting the choices of singular points.

By abuse of notation, we denote the pull-backs of $[\mathcal{Q}]$ and $[\mathcal{L}]$ to $H^2(\widetilde{\mathcal{C}}_\Delta, \mathbb Z/ 2\mathbb Z)$ in the same notations. Then, we have 
\[
H^1(\widetilde{\mathcal C}_{\Delta}, \mathbb Z / m \mathbb Z) = 0,\quad  H^2(\widetilde{\mathcal C}_\Delta, \mathbb Z / m \mathbb Z) = \langle H, h_p, h_q, E \rangle,
\]
where $H$ is the relative polarization pulled back from $|\mathcal O_{\mathbb P^2}(3)|,$ $h_p$ is the pullback of the ample generator on the copy of $\mathbb P^2$ parametrizing the (generically) non-singular point, $h_q$ is the pullback of the ample generator on the copy of $\mathbb P^2$ parametrizing the singular point, and $E$ is the pullback of the exceptional divisor on the blowup.

Since $\widetilde{\mathcal C}_{\Delta} \to \mathcal C_\Delta$ is a homeomorphism away from $\widetilde{\mathcal C}_{\Delta_{\geq 2}} \cup E$, we have
$$H^1(\mathcal C_{\Delta}, \mathbb Z / m \mathbb Z) \simeq \ker(H_{\mathrm{top}}(\widetilde{\mathcal C}_{\Delta_{\geq 2}} \cup E, \mathbb Z / m \mathbb Z) \to H_{\mathrm{top}}(\mathcal C_{\Delta_{\geq 2}} \cup \Sigma, \mathbb Z / m \mathbb Z) \oplus H^2(\widetilde{\mathcal C}_{\Delta}, \mathbb Z / m \mathbb Z)),$$
cf.\ the proof of Proposition~\ref{prop: brauer of symm},
where we recall that $\Sigma \subset \mathcal C_{\Delta}$ denotes the locus of singular points in the fibers. As before, the map $\widetilde{\mathcal{C}}_{\Delta_{\geq 2}} \cup E \to \mathcal{C}_{\Delta \geq 2} \cup \Sigma$ is generically $2$-to-$1$ on each irreducible component, so $\mathrm{Br}(\mathcal{C}_B)$ can only have $2$-primary torsion elements.

In order to compute the presentation of the classes $[\mathcal Q]$ and $[\mathcal L]$ with respect to the basis $$\langle H, h_p, h_q, E \rangle = H^2(\widetilde{\mathcal C}_{\Delta}, \mathbb Z / m \mathbb Z),$$ we consider the following pencils in $\widetilde{\mathcal C}_\Delta$:
\begin{enumerate}
\item Fix five points $p_1, \dots, p_5 \in \mathbb P^2$ and $q \in \mathbb P^2$ in general position. This datum uniquely determines a pencil of plane cubic curves that pass through the five points $p_1, \dots, p_5$ and admit a singularity at the point $q$. By marking the point $p_1$, we obtain a map
$$\varphi_H \colon \mathbb P^1 \to \widetilde{\mathcal C}_\Delta$$
satisfying
$$\deg \varphi_H^* H = 1,\;\;\; \deg \varphi_H^* h_p = \deg \varphi_H^* h_q = \deg \varphi_H^* E = 0.$$
The pencil contains $5$ reducible members that split as a conic through four of the five points and $q$ and the line through $q$ and the other point. In particular, we have
\[
\deg \varphi^*_H [\mathcal Q] = 4, \quad \deg \varphi_H^* [\mathcal L] = 1.
\]
\item Let $p_1,\ldots, p_5$ and $q$ be as above and fix a general line $l$ in $\mathbb P^2$. For any $p\in l$, there exists a unique cubic curve passing through $p_1,\ldots, p_5$ and $p$, as well as having a singularity at $q$. Marking the points $p$ and $q$, we obtain a pencil
$$\varphi_p \colon \mathbb P^1 \to \widetilde{\mathcal C}_\Delta$$
satisfying
$$\deg \varphi_p^*H = 3, \;\;\; \deg \varphi_p^*h_p = 1,  \;\;\; \deg \varphi^*_p h_q = \deg \varphi^*_p E = 0.$$
By considering the intersection between the line and the $5$ reducible members of the pencil, we obtain 
$$\deg \varphi_p^* [\mathcal Q] = 5 \cdot 2, \;\;\;\deg \varphi_p^*[\mathcal{L}] = 5.$$
\item Again, we take the same pencil, but fix a line in $\mathbb P^2$ that passes through $q$. Then, we mark the residual intersection point of this line with each member of this pencil and obtain a family
$$\varphi_E \colon \mathbb P^1 \to \widetilde{\mathcal C}_\Delta$$
satisfying
$$\deg \varphi^*_E H = \deg \varphi^*_E E = \deg \varphi^*_E h_p = 1, \;\;\; \deg \varphi_E^* h_q = 0.$$
By intersecting the $5$ reducible members of the pencil with the fixed line through $q$, we obtain
$$\deg \varphi_E^* [\mathcal Q] =  5, \;\;\; \deg \varphi^*_E [\mathcal L] = 0.$$

\item Finally, we take a different pencil as follows. Fix six points $p_1, \ldots, p_6 \in \mathbb P^2$ and a line $l$ in general position. For each $q\in l$, there exists a unique cubic curve passing through $p_1,\ldots, p_6$ and having a singularity at $q$. By marking the points $p_1$ and $q$, we obtain a map
\[
\varphi_q \colon \mathbb P^1 \to \widetilde{\mathcal{C}}_\Delta.
\]
Using classical geometry, one computes
\[
\deg \varphi_q^* H = 6, \quad \deg \varphi_q^* h_p = 0, \quad \deg \varphi_q^* h_q = 1,\quad \deg \varphi_q^* E = 0.
\]
Reducible cubics in this pencil have two types: the line $L$ passes through either one or two points among $p_1,\ldots, p_6$. From this, we obtain
\[
\deg \varphi^*_q [\mathcal Q] = 20, \quad \deg \varphi_q^* [\mathcal L] = 7.
\]

\end{enumerate}
Solving the resulting linear system of equations, it follows that 
\[
[\mathcal{Q}] = 4 H -2 h_p - 4 h_q +3 E, \quad [\mathcal{L}] = H + 2 h_p + h_q - 3E.
\]
Thus for $m=2$, the kernel we wish to compute equals $\mathbb Z/2\mathbb Z\cdot([\mathcal{Q}]-E)$, and $\mathrm{Br}(\mathcal{C}_B)[2]\simeq \mathbb Z/2\mathbb Z$. Repeating the argument with $m = 2^k$, we conclude that $\Br(\mathcal C_B) \simeq \mathbb Z/2\mathbb Z$.
\end{proof}

\begin{cor}
    Let $\mathcal C_B \to B \subset |\mathcal O_{\mathbb P^2}(3)|$ denote the universal smooth plane cubic. Then
    $$\Br(\mathcal C_B^{(2)}) = \mathbb Z/2\mathbb Z.$$
\end{cor}
\begin{proof}
    The claim follows from the previous lemma by observing that the Abel--Jacobi map
    \begin{align*}
    \mathcal C^{(2)}_B \to \Pic^{-2}_{\mathcal C / B} \simeq \Pic^1_{\mathcal C/B} \simeq \mathcal C_B
    \end{align*}
    is a $\mathbb P^1$-bundle.
\end{proof}

We are now ready to prove Proposition~\ref{prop: appendix}.
\begin{proof}[Proof of Proposition~\ref{prop: appendix}]
Let us first show that there is an isomorphism $$\Br(\mathcal C^{(3)}_B) \simeq H^1(B, \Pic_{\mathcal C/B}).$$
Indeed, let $\alpha \in \Br(\mathcal J)$ denote the Brauer class associated to the Brauer--Severi variety
$$\mathcal C_B^{(3)} \to \mathcal J,\quad Z \mapsto \mathcal I_Z(1).$$
By Lemma~\ref{lem: computing period}, we have $\operatorname{per}(\alpha) = 3$ and thus there is a short exact sequence
\begin{equation}\label{eq:ex1}0 \to \mathbb Z / 3 \mathbb Z \to \Br(\mathcal J) \to \Br(\mathcal C_B^{(3)}) \to 0.\end{equation}
On the other hand, the Leray spectral sequence yields an isomorphism
$$\Br(\mathcal J) \simeq H^1(B, \Pic_{\mathcal J/ B}),$$
whereas as the short exact sequence
$$0 \to \Pic^0_{\mathcal C / B} \to \Pic_{\mathcal J/B} \to \mathbb Z \to 0$$
yields an isomorphism $H^1(B, \Pic_{\mathcal J/B}) \simeq H^1(B, \Pic^0_{\mathcal C / B})$ and thus we have
$$H^1(B, \Pic^0_{\mathcal C / B}) \simeq H^1(B, \mathrm{Pic}_{\mathcal J / B}) \simeq \Br(\mathcal J).$$
Since the relative degree of any line bundle on $\mathcal C / B$ is divisible by three, the short exact sequence
$$0 \to \Pic^0_{\mathcal C / B} \to \Pic_{\mathcal C / B} \to \mathbb Z \to 0$$
induces the short exact sequence
\begin{equation}\label{eq:ex2}0 \to \mathbb Z / 3 \mathbb Z \to H^1(B, \Pic^0_{\mathcal C/ B})\simeq \Br(\mathcal J) \to H^1(B, \Pic_{\mathcal C / B}) \to H^1(B, \mathbb Z) = 0.\end{equation}
By comparing the short exact sequences (\ref{eq:ex1}) and (\ref{eq:ex2}), and using that $2 \cdot \Br(\mathcal C^{(3)}_B) = 0$ by Lemma~\ref{lem:  brauersymm2}, it follows that
$$\Br(\mathcal C^{(3)}_{B}) \simeq H^1(B,  \Pic_{\mathcal C/ B}),$$ and in particular $2 \cdot H^1(B, \mathrm{Pic}_{\mathcal C/B}) = 0$.

Let $\pi \colon \mathcal C_B \to B$ and $\pi^{(3)} \colon \mathcal C^{(3)}_B \to B$ denote the universal smooth plane cubic and its relative symmetric product. The functoriality of the Grothendieck spectral sequence with respect to the natural morphism
\begin{align*}R\pi_* \mathbb G_m &\to R\pi^{(3)}_* \mathbb G_m \\
\gamma &\mapsto \gamma^{(3)}
\end{align*}
induces the diagram
$$\begin{tikzcd}
  0 = \Br(B) \arrow[r] \arrow[d] &\Br(\mathcal C_B)\simeq \mathbb Z/2\mathbb Z \arrow[r] \arrow[d] & {H^1(B, \Pic_{\mathcal C/B})} \arrow[d] \arrow[r] & {H^3(B, \mathbb G_m)} \arrow[d, "\times 3"] \\
0 = \Br(B) \arrow[r] & F^1 \mathrm{Br}(\mathcal C_B^{(3)}) \arrow[r] & {H^1(B, \Pic_{\mathcal C^{(3)}/B})} \arrow[r]              & {H^3(B, \mathbb G_m)}                     
\end{tikzcd}$$
with exact rows, where $\Br(\mathcal C_B) \simeq \mathbb Z / 2 \mathbb Z$ by Lemma~\ref{lem: appendixc0}. Note that intersecting with a general line induces a section of $\mathcal C^{(3)}_B \to B$, so the map $H^1(B, \Pic_{\mathcal C^{(3)} /B}) \to H^3(B, \mathbb G_m)$ is trivial. Since we have $2 \cdot H^1(B, \Pic_{\mathcal C /B}) = 0$ by the first part of this proof, a simple diagram chase then shows that the map
$$\mathbb Z / 2 \mathbb Z \simeq \Br(\mathcal C_B) \to H^1(B, \Pic_{\mathcal C / B})$$
is an isomorphism.
All in all, we conclude that
\[
\Br(\mathcal C^{(3)}_B) \simeq H^1(B, \Pic_{\mathcal C/B}) = \mathbb Z / 2 \mathbb Z.
\qedhere
\]
\end{proof}

Finally, we finish the proof of Theorem~\ref{thm: tate-shafarevich} for $d = 3$.
\begin{thm}
The Tate--Shafarevich group of the relative Jacobian of the universal smooth plane cubic curve is cyclic of order $6$, i.e.,
$$\Sha(\mathcal J) \simeq \mathbb Z / 3 \mathbb Z \cdot [\Pic^1_{\mathcal C / B}] \oplus \mathbb Z / 2 \mathbb Z.$$
\end{thm}
\begin{proof}
    We follow the argument in the proof of Theorem~\ref{thm: tate-shafarevich main text}. Using that
    \[
    \mathrm{Br}(B) = 0, \quad H^0(B, \NS_{\mathcal{J} / B}) / H^0(B, \Pic_{\mathcal{J}/B}) = 0,
    \]
    and note that $F^1\mathrm{Br}(\mathcal{J}) = \Br(\mathcal{J})$ in the $d=3$ case, we have
    \[
    \Sha(\mathcal{J}) \simeq H^1(B, \mathrm{Pic}_{\mathcal{J}/B}) \simeq \mathrm{Br}(\mathcal{J}).
    \]
    Now we have the short exact sequence
    $$0 \to \langle \alpha\rangle \to \Br(\mathcal J) \to \Br(\mathcal C^{(3)}_B) \to 0,$$
    so the claim that $\mathrm{Br}(\mathcal{J}) \simeq \mathbb Z/6\mathbb Z$ follows from Lemma~\ref{lem: computing period} and Proposition~\ref{prop: appendix}.
\end{proof}

\section{Primitively polarized K3 surfaces}\label{sec: K3}

The purpose of this section is to prove Theorem~\ref{thm: brauer k3}, which we recall below, and to exhibit concrete counterexamples for low degree polarized K3 surfaces.

\begin{thm}[Theorem~\ref{thm: brauer k3}]
    Let $(S, L)$ be a primitively polarized K3 surface of Picard rank one and degree at least $10$. Then
    \[
    \mathrm{Br}(|L|_\mathrm{sm}) = 0.
    \]
\end{thm}

\begin{proof}
    We first verify the ampleness condition \ref{defn: ampleness} for $L$. It is shown in \cite{Knutsen} that $L^2 \geq 4$ implies very ampleness on such K3 surfaces. The second condition in (ii) is automatic since every curve in $|L|$ is integral. For the rest of (ii) and (iii), we follow the strategy described in Remark~\ref{rmk: strategy to check dagger} and argue that the $k$-nodal loci for $k=1, 2$ are irreducible and dense in the $\{\delta = k\}$ loci, and of the expected dimension $\dim |L| - k$. Indeed, these results can be found in \cite[Section~4]{Dedieu-Sernesi} and \cite[Theorem~1.2]{Bruno-Lelli-Chiesa}; see also \cite[Section~2]{Ciliberto-Dedieu}. 
    
    For (iv), we use a monodromy argument involving tacnodes. For $L^2 \geq 4$, it is proved in \cite[Corollary~4.2]{Galati-Knutsen} that the 2-nodal locus in $|L|$ contains a non-empty, generically smooth divisor $\Delta_{\mathrm{tac}}$ consisting of curves with a tacnode. Furthermore, for $[C]\in \Delta_{\mathrm{tac}}$, the differential map 
    \[
    H^0(C, N_{C/S}) \to H^0(C, T_C^1)
    \]
    of the versal map at $[C]$ is surjective, see \cite[Theorem~1.1]{Galati-Knutsen}. It follows that the local deformation family of $[C]\in |D|$ maps to the versal deformation family of $C$ as a submersion. Thus we could choose a small disk $\mathbb D$ in $|D|$ centered at $[C]$, and local coordinates $x, y$ such that the curves $C_t$ are given by the equations $y^2=(x^2-t)^2$ for $t\in \mathbb D$, cf.~\cite[Section~2.4.1]{Caporaso-Harris}. Now choosing a loop $\gamma \subset \mathbb D$ around the origin, it is clear that circling around $\gamma$ exchanges the two nodes for a 2-nodal curve $C_t$. Since the 2-nodal locus is irreducible, this implies that the universal $2$-nodal locus is also irreducible, and the same holds for its closure $\overline{\Sigma_2}$. 
    
    Now we apply Theorem~\ref{thm:brauer_base} to $(S,L)$. Recall from Lemma~\ref{lemma: relative degree} that
    \[
    \deg_{\Sigma/S} \Sigma_{\geq 2} = e(S) - 7 + 3 (L^2) + 2 (K_S \cdot L),
    \]
    which for a K3 surface is always odd (!). Thus $[\overline{\Sigma_{2}}]\neq 0 \in H^2(\Sigma, \mathbb Z/ 2\mathbb Z)$, and the claim follows.
\end{proof}

The rest of this section is devoted to studying two K3 surfaces of low degrees and the corresponding Brauer group. In particular, we shall see that Theorem~\ref{thm: brauer k3} fails for them.

\subsection{Degree $2$.} Consider a double cover $f: S \to \mathbb P^2$ branched at a general smooth sextic curve $E$, i.e.~a K3 surface of Picard rank one and degree two, polarized by the pull-back $L$ of the hyperplane class $\mathcal O(1)$ on $\mathbb P^2$. 
We first observe that there is a homeomorphism between the following two spaces:
\[
\Sigma = \{(C, p) \mid p \in \mathrm{Sing}(C)\} \subset |L| \times S
\]
and 
\[
\{(l, q) \mid \mathrm{mult}_q(l \cap E) \geq 2\} \subset |\mathcal O(1)| \times \mathbb P^2,
\]
via the push-forward by $f$. Note that the latter space is nothing but $E$, via the homeomorphism sending $q \in E$ to $(T_q E, q)$. It follows that $\Sigma \simeq E$ as topological spaces. One sees also $\Delta \simeq E^\vee$, the dual curve of $E$.

A curve $C$ in $|L|$ can have at most three isolated singularities. We thus have the stratification
\[\begin{tikzcd}
	{\Sigma_{2} \sqcup \Sigma_3} & \Sigma & {\Sigma_1} \\
	{\Delta_{2} \sqcup \Delta_3} & \Delta & {\Delta_1}
	\arrow[hook, from=1-1, to=1-2]
	\arrow[from=1-1, to=2-1]
	\arrow[from=1-2, to=2-2]
	\arrow[hook', from=1-3, to=1-2]
	\arrow[from=1-3, to=2-3]
	\arrow[hook, from=2-1, to=2-2]
	\arrow[hook', from=2-3, to=2-2]
\end{tikzcd}\]
where $\Delta_k$ are now exactly the loci of curves with $k$ singular points. Note also that $\Delta_2$ and $\Delta_3$ consist of isolated closed points. Following Lemma~\ref{lemma: br isom to ker}, we have a short exact sequence
\[
0\to H_1(\Sigma) \to \mathrm{Br}(|L|_\mathrm{sm})[p] \to \ker\left(H_0(\Sigma_{2} \sqcup \Sigma_3) \to H_0 (\Sigma) \oplus H_0(\Delta_{2}\sqcup \Delta_3) \right) \to 0
\]
where the homology is taken with $\mathbb Z/p \mathbb Z$-coefficient. Now $H_1(\Sigma)\simeq H_1(E)$ is 20-dimensional as a $\mathbb F_p$-vector space; moreover, the $\mathbb F_p$-dimension of the kernel is at least
\[
|\Sigma_2 \sqcup \Sigma_3| - |\Delta_2 \sqcup \Delta_3| - 1 = |\Delta_2| + 2 |\Delta_3| - 1.
\]
Now it is classical that there are exactly $324$ bitangents of $E$, and accordingly $324$ binodal curves in $\Delta_2$. We thus see that $\mathrm{Br}(|L_\mathrm{sm}|)[p]$ is at least $343$-dimensional for every prime $p$. 

\subsection{Degree $4$.} \label{sec: degree 4 K3} Consider a general smooth quartic surface in $\mathbb P^3$, i.e.~a K3 surface of Picard rank one and degree four, polarized by the pull-back $L$ of the hyperplane class $\mathcal{O}(1)$ on $\mathbb P^3$. We verify first that the ampleness condition \ref{defn: ampleness} is actually satisfied. Assumptions (i) and the second part of (ii) are automatic. For the rest, we argue once again that the $k${-nodal} loci for $k=1,2$ are irreducible and dense in the $\{\delta = k\}$ loci. The denseness statement is \cite[Corollary~1.2]{ChenXi}. The irreducibility for $k=1$ follows from that of the discriminant divisor, while for $k=2$ it is \cite[Theorem~9.2]{Ciliberto-Dedieu-Limits}. Finally, (iv) is given by \cite{doublepointcurve}. Thus $\mathrm{Br}(|L|_\mathrm{sm})$ is described by Theorem~\ref{thm:brauer_base}. Note, however, that Lemma~\ref{lemma: relative degree} does not apply since $\dim |L| = 3$.

We now study the universal singular locus $\Sigma$. If $p$ is a singular point on a hyperplane section $C = S \cap H$, then we have
\[
\dim T_p C \geq 2 \textnormal{ \:\:and \:} H \supset T_p C \subset T_p S.
\]
It follows that $H = T_p S$, and thus the map
\[
S \to \Sigma, \quad p \mapsto (p, (T_p S) \cap S)
\]
is an isomorphism. To describe the universal $2$-nodal locus, we recall the Gauss map
\[
\gamma: S \to (\mathbb P^3)^\vee, \quad p \mapsto T_p S.
\]
Two nodes $p, q \in \Sigma$ lie on the same hyperplane section if and only if $\gamma(p) = \gamma(q)$. Let $\mathbb D(\gamma) \in \mathrm{CH}_1(S)$ be the \textit{double point class} in the sense of \cite[Chapter~9.3]{Fulton}, viewed as a cohomology class, and let $R(\gamma)$ be the ramification locus of $\gamma$. Note that a general quartic K3 surface has only finitely many (in fact, $3200$) tri-tangent planes, so $\Sigma_2$ maps birationally to the double point set in $S$ (in the terminology of \cite{Fulton}). We thus have
\[
\mathbb D(\gamma) = [\overline{\Sigma_2}] + 2 R(\gamma).
\]

Now by loc.~cit.~$\mathbb D(\gamma)$ is computed by the formula
\begin{align*}
\mathbb D(\gamma) &= \gamma^* {\gamma}_*[S] - \big(c(\gamma^* T_{\mathbb {P}^3}) \cdot c(T_S)^{-1}\big)_1 \cap [S] = 108 L - 12 L = 96 L,
\end{align*}
where we used that the degree of $S^\vee$ is $4\cdot 3^2 = 36$, and that the Gauss map is given by cubics. The ramification locus is given by the Hessian matrix of the defining equation of $S$ and equals $8 L$. We conclude that $[\overline{\Sigma_2}] = 80 L$
and thus vanishes in $H^2(\Sigma, \mathbb Z/ 2 \mathbb Z)$. Theorem~\ref{thm:brauer_base} then gives
\[
\mathrm{Br}(|L|_\mathrm{sm}) \simeq \mathbb Z/ 2\mathbb Z.
\]

\begin{rmk}\label{rmk: lagrangian smooth locus}
    For a polarized K3 surface $(S,L)$ as in Theorem~\ref{thm: brauer k3}, if we pick a suitable Mukai vector ${v}\in H^*(S, \mathbb Z)$, the linear system $|L|$ can be viewed as the base of a Lagrangian fibration
\[
h: M_L({v}) \to |L|
\]
of K3$^{[n]}$-type, where $M_L(v)$ is the moduli space of (Gieseker) $L$-semistable sheaves on $S$ with Mukai vector $v$. Denote by $h_\mathrm{sm}\subset |L|$ the smooth locus of $h$; this contains $|L|_\mathrm{sm}$\footnote{In fact, one can show that $h_{\mathrm{sm}} = |L|_{\mathrm{sm}}$, but this is not needed for the argument.} and we have inclusions
\[
\mathrm{Br}(h_\mathrm{sm}) \hookrightarrow \mathrm{Br}(|L|_\mathrm{sm}) \hookrightarrow \mathrm{Br}(K(|L|)).
\]
Theorem~\ref{thm: brauer k3} thus implies $\mathrm{Br}(h_\mathrm{sm}) = 0$. It is an interesting question\footnote{{We thank Daniel Huybrechts for asking it, which led us to consider cubic fourfolds and the LSV fibration.}} whether this vanishing holds for general Lagrangian fibrations (assuming the dimension of the base is sufficiently large to avoid the counterexamples above). We study this for the Laza--Saccà--Voisin fibration of OG10 type \cite{LSV} in the next section.
\end{rmk}

\section{Cubic fourfolds and the smooth locus of the LSV fibration} \label{sec: cubic fourfold}

We prove Theorem~\ref{thm: brauer cubic fourfold} in this section. Let $X$ be a general smooth cubic fourfold and let $L = \mathcal{O}_X(1)$. Note that, as in Section~\ref{sec: degree 4 K3}, the discriminant locus $\Delta \subset |\mathcal O_X(1)| \simeq (\mathbb P^5)^\vee$ is the dual variety of $X$, i.e., the image of $X$ under the Gauss map
\begin{align*}
    \gamma_X \colon X &\to (\mathbb P^5)^\vee \\
    p &\mapsto T_p X.
\end{align*}
\begin{lem}\label{lem: dagger cubic}
    The line bundle $L$ satisfies the condition \ref{defn: ampleness}.
\end{lem}
\begin{proof}
By definition $L$ is very ample, and condition (ii) in \ref{defn: ampleness} follows from the fact that $\gamma_X$ is finite and birational; see \cite[Corollary~10.26]{3264andallthat}.

For (iii), we claim that $\Delta_{\geq 2}$ agrees with the image in $(\mathbb P^5)^\vee$ of the locus of points of $X$ contained in lines of the second type, i.e., lines on which the restriction of the Gauss map $\gamma_X$ is $2$-to-$1$; see \cite[Section~2.2]{geometryofcubics}. Indeed, let $H \subset \mathbb P^5$ denote a hyperplane such that $X \cap H$
 has at least two distinct singular points $p_1, p_2 \in \Sing(X \cap H)$. Then, the line $l \subset \mathbb P^5$ through $p_1$ and $p_2$ intersects $X$ with degree at least $4$ and is thus contained in $X$. Furthermore, we have $\gamma_X(p_1) = \gamma_X(p_2) = H$ and thus the restriction of the Gauss map $\gamma_X|_l$ is of degree two, i.e., $l$ is a line of the second type on $X$.
 
 Conversely, if $l$ is a line of the second type and $p_1 \in l$ is a general point, then there is another point $p_2 \in l$ with $\gamma_X(p_1) = \gamma_X(p_2).$ The associated hyperplane section is thus singular at $p_1$ and $p_2$ and corresponds to a point in $\Delta_{\geq 2}$.
 
 All in all, we see that $\Delta_{\geq 2}$ agrees with the image in $(\mathbb P^5)^\vee$ of the locus $W\subset X$ of points contained in lines of the second type. By \cite[Corollary~2.2.14]{geometryofcubics}, this locus is irreducible of dimension $3$ for a general cubic fourfold. It follows from \cite[Corollary~3.7]{gounelaskouvidakis} that the general hyperplane section with more than one singularity has exactly two singular points, hence $\Delta_2 \subset \Delta_{\geq 2}$ is a dense subset and condition (iii) in \ref{defn: ampleness} is satisfied. The above argument also shows that $\Sigma_{\geq 2} = W$; it is irreducible, which gives (iv).
\end{proof}

\begin{thm}[{Theorem~\ref{thm: brauer cubic fourfold}}]\label{thm: cubic fourfold main text}
    Let $X$ be a general cubic fourfold. Then, we have
    $$\Br(|\mathcal O_X(1)|_{\mathrm{sm}}) = 0.$$
\end{thm}
\begin{proof}
    By Lemma \ref{lem:brauer_bm}, it suffices to show
    $$H_{\mathrm{top}-1}(\gamma_X(X), \mathbb Z / p \mathbb Z) = 0$$
    for all primes $p$. 
    By the proof of Lemma \ref{lem: dagger cubic}, the restriction of the Gauss map to the complement of the set $W \subset X$ of points contained in lines of the second type is bijective. Since $H^1(X, \mathbb Z / p \mathbb Z) = 0$, we have
    $$H_{\mathrm{top}-1}(\gamma_X(X), \mathbb Z /p \mathbb Z) \simeq \ker(H_{\mathrm{top}}(W, \mathbb Z / p \mathbb Z) \to H^2(X, \mathbb Z / p \mathbb Z) \oplus H_{\mathrm{top}}(\gamma_X(W), \mathbb Z / p \mathbb Z)).$$
    Since $W$ is irreducible and $\gamma_X|_W \colon W\to \gamma_X(W)$ is generically of degree two,
    it remains to show that the class $[W] \in H^2(X, \mathbb Z/2 \mathbb Z) \cong \mathbb Z / 2\mathbb Z \cdot [\mathcal O_X(1)]$ is non-zero. By \cite[Lemma~2.6~(8)]{gounelaskouvidakis}, we have $[W] = 75 [\mathcal O_X(1)]$, and hence
    $$\Br(|\mathcal O_X(1)|_{\mathrm{sm}})[p]\simeq H_{\mathrm{top}-1}(\gamma_X(X), \mathbb Z / p \mathbb Z) =  0$$
    for all primes $p$ as desired.
\end{proof}

\begin{rmk}
    By \cite{LSV}, the intermediate Jacobian fibration associated to the cubic fourfold $X$ can be compactified into a hyperkähler variety $\overline{\mathcal{J}}$ of OG10 type, with a Lagrangian fibration 
    \[
    h: \overline{\mathcal{J}} \to |\mathcal{O}_X(1)|.
    \]
    See also \cite{Bottini} for a modular description of $\overline{\mathcal{J}}$ parametrizing semistable torsion sheaves on the Fano variety of lines of $X$. Theorem~\ref{thm: cubic fourfold main text} thus implies the smooth locus of $h$ has trivial Brauer group, cf.~Remark \ref{rmk: lagrangian smooth locus}. Similar to the situation there, one can show that the smooth locus of $h$ coincides with the smooth locus of the linear system $|\mathcal O_X(1)|,$ but we only need the inclusion $|\mathcal O_X(1)|_{\mathrm{sm}} \subset h_{\mathrm{sm}}$, which is immediate from the construction in \cite{LSV}.
\end{rmk}

\printbibliography

@article {Knutsen,
    AUTHOR = {Knutsen, Andreas Leopold},
     TITLE = {On {$k$}th-order embeddings of {$K3$} surfaces and {E}nriques
              surfaces},
   JOURNAL = {Manuscripta Math.},
  FJOURNAL = {Manuscripta Mathematica},
    VOLUME = {104},
      YEAR = {2001},
    NUMBER = {2},
     PAGES = {211--237},
      ISSN = {0025-2611,1432-1785},
   MRCLASS = {14C20 (14J28)},
  MRNUMBER = {1821184},
MRREVIEWER = {Hiroyuki\ Terakawa},
       DOI = {10.1007/s002290170040},
       URL = {https://doi.org/10.1007/s002290170040},
}

@article {AC,
    AUTHOR = {Arbarello, Enrico and Cornalba, Maurizio},
     TITLE = {A few remarks about the variety of irreducible plane curves of
              given degree and genus},
   JOURNAL = {Ann. Sci. \'Ecole Norm. Sup. (4)},
  FJOURNAL = {Annales Scientifiques de l'\'Ecole Normale Sup\'erieure.
              Quatri\`eme S\'erie},
    VOLUME = {16},
      YEAR = {1983},
    NUMBER = {3},
     PAGES = {467--488},
      ISSN = {0012-9593},
   MRCLASS = {14H10},
  MRNUMBER = {740079},
MRREVIEWER = {A.\ Hirschowitz},
       URL = {http://www.numdam.org/item?id=ASENS_1983_4_16_3_467_0},
}

@article {Z1,
    AUTHOR = {Zariski, Oscar},
     TITLE = {Dimension-theoretic characterization of maximal irreducible
              algebraic systems of plane nodal curves of a given order
              {$n$}\ and with a given number {$d$}\ of nodes},
   JOURNAL = {Amer. J. Math.},
  FJOURNAL = {American Journal of Mathematics},
    VOLUME = {104},
      YEAR = {1982},
    NUMBER = {1},
     PAGES = {209--226},
      ISSN = {0002-9327,1080-6377},
   MRCLASS = {14N10 (14H10 14N05)},
  MRNUMBER = {648487},
MRREVIEWER = {Abramo\ Hefez},
       DOI = {10.2307/2374074},
       URL = {https://doi.org/10.2307/2374074},
}

@article {Harris,
    AUTHOR = {Harris, Joe},
     TITLE = {On the {S}everi problem},
   JOURNAL = {Invent. Math.},
  FJOURNAL = {Inventiones Mathematicae},
    VOLUME = {84},
      YEAR = {1986},
    NUMBER = {3},
     PAGES = {445--461},
      ISSN = {0020-9910,1432-1297},
   MRCLASS = {14H10 (14H20)},
  MRNUMBER = {837522},
MRREVIEWER = {R.\ F.\ Lax},
       DOI = {10.1007/BF01388741},
       URL = {https://doi.org/10.1007/BF01388741},
}

@article {Roberts,
    AUTHOR = {Roberts, M. Samuel},
     TITLE = {Sur l'ordre des conditions de la coexistence des \'equations
              alg\'ebriques \`a{} plusieurs variables},
   JOURNAL = {J. Reine Angew. Math.},
  FJOURNAL = {Journal f\"ur die Reine und Angewandte Mathematik. [Crelle's
              Journal]},
    VOLUME = {67},
      YEAR = {1867},
     PAGES = {266--278},
      ISSN = {0075-4102,1435-5345},
   MRCLASS = {99-04},
  MRNUMBER = {1579373},
       DOI = {10.1515/crll.1867.67.266},
       URL = {https://doi.org/10.1515/crll.1867.67.266},
}

@article {Vainsencher,
    AUTHOR = {Vainsencher, Israel},
     TITLE = {Counting divisors with prescribed singularities},
   JOURNAL = {Trans. Amer. Math. Soc.},
  FJOURNAL = {Transactions of the American Mathematical Society},
    VOLUME = {267},
      YEAR = {1981},
    NUMBER = {2},
     PAGES = {399--422},
      ISSN = {0002-9947,1088-6850},
   MRCLASS = {14N10},
  MRNUMBER = {626480},
MRREVIEWER = {Ragni\ Piene},
       DOI = {10.2307/1998661},
       URL = {https://doi.org/10.2307/1998661},
}

@article {Dedieu-Sernesi,
    AUTHOR = {Dedieu, T. and Sernesi, E.},
     TITLE = {Equigeneric and equisingular families of curves on surfaces},
   JOURNAL = {Publ. Mat.},
  FJOURNAL = {Publicacions Matem\`atiques},
    VOLUME = {61},
      YEAR = {2017},
    NUMBER = {1},
     PAGES = {175--212},
      ISSN = {0214-1493,2014-4350},
   MRCLASS = {14C20 (14H10 14H20 14J26 14J28)},
  MRNUMBER = {3590119},
MRREVIEWER = {Yusuf\ Mustopa},
       DOI = {10.5565/PUBLMAT\_61117\_07},
       URL = {https://doi.org/10.5565/PUBLMAT_61117_07},
}

@misc{Bruno-Lelli-Chiesa,
 author = {Andrea Bruno and Margherita Lelli-Chiesa},
 title = {Irreducibility of {Severi} varieties on {K3} surfaces},
 year = {2021},
 howpublished = {Preprint, {arXiv}:2112.09398},
 url = {https://arxiv.org/abs/2112.09398},
 arXiv = {arXiv:2112.09398}
}

@article {Ciliberto-Dedieu,
    AUTHOR = {Ciliberto, Ciro and Dedieu, Thomas},
     TITLE = {On the irreducibility of {S}everi varieties on {$K3$}
              surfaces},
   JOURNAL = {Proc. Amer. Math. Soc.},
  FJOURNAL = {Proceedings of the American Mathematical Society},
    VOLUME = {147},
      YEAR = {2019},
    NUMBER = {10},
     PAGES = {4233--4244},
      ISSN = {0002-9939,1088-6826},
   MRCLASS = {14C20 (14J28)},
  MRNUMBER = {4002538},
MRREVIEWER = {Yusuf\ Mustopa},
       DOI = {10.1090/proc/14559},
       URL = {https://doi.org/10.1090/proc/14559},
}

@article {DiLorenzo-Pirisi,
    AUTHOR = {Di Lorenzo, A. and Pirisi, R.},
     TITLE = {The {B}rauer groups of moduli of genus three curves, abelian
              threefolds and plane curves},
   JOURNAL = {Compos. Math.},
  FJOURNAL = {Compositio Mathematica},
    VOLUME = {161},
      YEAR = {2025},
    NUMBER = {7},
     PAGES = {1664--1697},
      ISSN = {0010-437X,1570-5846},
   MRCLASS = {14F22 (14H10)},
  MRNUMBER = {4961003},
MRREVIEWER = {Igor\ A.\ Rapinchuk},
       DOI = {10.1112/S0010437X25007481},
       URL = {https://doi.org/10.1112/S0010437X25007481},
}

@book {Ciliberto-Dedieu-Limits,
    AUTHOR = {Ciliberto, Ciro and Dedieu, Thomas},
     TITLE = {Limits of pluri-tangent planes to quartic surfaces, \textup{Algebraic and complex geometry}},
 BOOKTITLE = {Algebraic and complex geometry},
    SERIES = {Springer Proc. Math. Stat.},
    VOLUME = {71},
     PAGES = {123--199},
 PUBLISHER = {Springer, Cham},
      YEAR = {2014},
      %ISBN = {978-3-319-05404-9; 978-3-319-05403-2},
   MRCLASS = {14J28 (14B07)},
  MRNUMBER = {3278573},
MRREVIEWER = {Trygve\ Johnsen},
       DOI = {10.1007/978-3-319-05404-9\_6},
       URL = {https://doi.org/10.1007/978-3-319-05404-9_6},
}

@article {ChenXi,
    AUTHOR = {Chen, Xi},
     TITLE = {Nodal curves on {K}3 surfaces},
   JOURNAL = {New York J. Math.},
  FJOURNAL = {New York Journal of Mathematics},
    VOLUME = {25},
      YEAR = {2019},
     PAGES = {168--173},
      ISSN = {1076-9803},
   MRCLASS = {14J28 (14E05)},
  MRNUMBER = {3904883},
MRREVIEWER = {Kenta\ Watanabe},
}

@book {Fulton,
    AUTHOR = {Fulton, William},
     TITLE = {Intersection theory},
    SERIES = {Ergebnisse der Mathematik und ihrer Grenzgebiete (3)},
    VOLUME = {2},
 PUBLISHER = {Springer-Verlag, Berlin},
      YEAR = {1984},
     PAGES = {xi+470},
      ISBN = {3-540-12176-5},
   MRCLASS = {14C17 (14-02 14C40)},
  MRNUMBER = {732620},
MRREVIEWER = {Werner\ Kleinert},
       DOI = {10.1007/978-3-662-02421-8},
       URL = {https://doi.org/10.1007/978-3-662-02421-8},
}

@article {LSV,
    AUTHOR = {Laza, Radu and Sacc\`a, Giulia and Voisin, Claire},
     TITLE = {A hyper-{K}\"ahler compactification of the intermediate
              {J}acobian fibration associated with a cubic 4-fold},
   JOURNAL = {Acta Math.},
  FJOURNAL = {Acta Mathematica},
    VOLUME = {218},
      YEAR = {2017},
    NUMBER = {1},
     PAGES = {55--135},
      ISSN = {0001-5962,1871-2509},
   MRCLASS = {53C26 (14J10 14J35 32Q15 32S99)},
  MRNUMBER = {3710794},
MRREVIEWER = {Grzegorz\ Kapustka},
       DOI = {10.4310/ACTA.2017.v218.n1.a2},
       URL = {https://doi.org/10.4310/ACTA.2017.v218.n1.a2},
}

@article {computingsym2011,
    AUTHOR = {Bernardi, Alessandra and Gimigliano, Alessandro and Id\`a,
              Monica},
     TITLE = {Computing symmetric rank for symmetric tensors},
   JOURNAL = {J. Symbolic Comput.},
  FJOURNAL = {Journal of Symbolic Computation},
    VOLUME = {46},
      YEAR = {2011},
    NUMBER = {1},
     PAGES = {34--53},
      ISSN = {0747-7171,1095-855X},
   MRCLASS = {14N05},
  MRNUMBER = {2736357},
       DOI = {10.1016/j.jsc.2010.08.001},
       URL = {https://doi.org/10.1016/j.jsc.2010.08.001},
}

@book {Markman,
    AUTHOR = {Markman, Eyal},
     TITLE = {Lagrangian fibrations of holomorphic-symplectic varieties of
              {$K3^{[n]}$}-type, \textup{Algebraic and complex geometry}},
 BOOKTITLE = {Algebraic and complex geometry},
    SERIES = {Springer Proc. Math. Stat.},
    VOLUME = {71},
     PAGES = {241--283},
 PUBLISHER = {Springer, Cham},
      YEAR = {2014},
      %ISBN = {978-3-319-05404-9; 978-3-319-05403-2},
   MRCLASS = {53C55 (14J28 32G20 32Q15 53D12)},
  MRNUMBER = {3278577},
MRREVIEWER = {Jing\ Zhang},
       DOI = {10.1007/978-3-319-05404-9\_10},
       URL = {https://doi.org/10.1007/978-3-319-05404-9_10},
}

@misc{HM2,
 author = {D. Huybrechts and D. Mattei},
 title = {The Tate--Shafarevich group of a polarised K3 surface},
 year = {2025},
 howpublished = {Preprint, arXiv:2507.22703},
 url = {https://arxiv.org/abs/2507.22703},
 arXiv = {arXiv:2507.22703}
}

@misc{Mattei-Meinsma,
 author = {D. Mattei and R. Meinsma},
 title = {Obstruction classes for moduli spaces of sheaves and Lagrangian fibrations},
 year = {2024},
 howpublished = {Preprint, arXiv:2404.16652v2},
 url = {https://arxiv.org/abs/2404.16652v2},
 arXiv = {arXiv:2404.16652v2}
}

@misc{Bottini,
 author = {Alessio Bottini},
 title = {O'Grady's tenfolds from stable bundles on hyper-Kähler fourfolds},
 year = {2024},
 howpublished = {Preprint, arXiv:2411.18528},
 url = {https://arxiv.org/abs/2411.18528},
 arXiv = {arXiv:2411.18528}
}

@misc{Gounelas-Huybrechts,
 author = {Frank Gounelas and Daniel Huybrechts},
 title = {Universal Brauer--Severi varieties},
 year = {2025},
 howpublished = {Preprint, arXiv:2510.20474},
 url = {https://arxiv.org/abs/2510.20474},
 arXiv = {arXiv:2510.20474}
}

@article {Vassiliev,
    AUTHOR = {Vassiliev, V. A.},
     TITLE = {How to calculate the homology of spaces of nonsingular
              algebraic projective hypersurfaces},
   JOURNAL = {Tr. Mat. Inst. Steklova},
  FJOURNAL = {Trudy Matematicheskogo Instituta Imeni V. A. Steklova},
    VOLUME = {225},
      YEAR = {1999},
     PAGES = {132--152},
      ISSN = {0371-9685,3034-1809},
   MRCLASS = {55R80 (57R19)},
  MRNUMBER = {1725936},
MRREVIEWER = {Sergei\ V.\ Chmutov},
}

@article {AM,
    AUTHOR = {Antieau, Benjamin and Meier, Lennart},
     TITLE = {The {B}rauer group of the moduli stack of elliptic curves},
   JOURNAL = {Algebra Number Theory},
  FJOURNAL = {Algebra \& Number Theory},
    VOLUME = {14},
      YEAR = {2020},
    NUMBER = {9},
     PAGES = {2295--2333},
      ISSN = {1937-0652,1944-7833},
   MRCLASS = {14F22 (11R37 14D23 14H10 14H52)},
  MRNUMBER = {4172709},
MRREVIEWER = {Fabian\ Reede},
       DOI = {10.2140/ant.2020.14.2295},
       URL = {https://doi.org/10.2140/ant.2020.14.2295},
}

@article {Orsola,
    AUTHOR = {Tommasi, Orsola},
     TITLE = {Stable cohomology of spaces of non-singular hypersurfaces},
   JOURNAL = {Adv. Math.},
  FJOURNAL = {Advances in Mathematics},
    VOLUME = {265},
      YEAR = {2014},
     PAGES = {428--440},
      ISSN = {0001-8708,1090-2082},
   MRCLASS = {14F25 (14J70 20G10)},
  MRNUMBER = {3255466},
MRREVIEWER = {Sergio\ L.\ Cacciatori},
       DOI = {10.1016/j.aim.2014.08.005},
       URL = {https://doi.org/10.1016/j.aim.2014.08.005},
}

@article {Collino,
    AUTHOR = {Collino, Alberto},
     TITLE = {The rational equivalence ring of symmetric products of curves},
   JOURNAL = {Illinois J. Math.},
  FJOURNAL = {Illinois Journal of Mathematics},
    VOLUME = {19},
      YEAR = {1975},
    NUMBER = {4},
     PAGES = {567--583},
      ISSN = {0019-2082},
   MRCLASS = {14C15 (14H40)},
  MRNUMBER = {389897},
MRREVIEWER = {Allen\ B.\ Altman},
       URL = {http://projecteuclid.org/euclid.ijm/1256050666},
}

@book {ACGH85,
    AUTHOR = {Arbarello, E. and Cornalba, M. and Griffiths, P. A. and
              Harris, J.},
     TITLE = {Geometry of algebraic curves:~Volume~I},
    SERIES = {Grundlehren der mathematischen Wissenschaften},
    VOLUME = {267},
 PUBLISHER = {Springer-Verlag, New York},
      YEAR = {1985},
     PAGES = {xvi+386},
      ISBN = {0-387-90997-4},
   MRCLASS = {14Hxx (14-02)},
  MRNUMBER = {770932},
MRREVIEWER = {Werner\ Kleinert},
       DOI = {10.1007/978-1-4757-5323-3},
       URL = {https://doi.org/10.1007/978-1-4757-5323-3},
}

@book {3264andallthat,
    AUTHOR = {Eisenbud, David and Harris, Joe},
     TITLE = {3264 and all that---a second course in algebraic geometry},
 PUBLISHER = {Cambridge University Press, Cambridge},
      YEAR = {2016},
     PAGES = {xiv+616},
      %ISBN = {978-1-107-60272-4; 978-1-107-01708-5},
   MRCLASS = {14-01 (14C15 14M15 14N10)},
  MRNUMBER = {3617981},
MRREVIEWER = {Arnaud\ Beauville},
       DOI = {10.1017/CBO9781139062046},
       URL = {https://doi.org/10.1017/CBO9781139062046},
}

@article {gounelaskouvidakis,
    AUTHOR = {Gounelas, Frank and Kouvidakis, Alexis},
     TITLE = {Geometry of lines on a cubic four-fold},
   JOURNAL = {Int. Math. Res. Not. IMRN},
  FJOURNAL = {International Mathematics Research Notices. IMRN},
      YEAR = {2024},
    NUMBER = {2},
     PAGES = {1373--1421},
      ISSN = {1073-7928,1687-0247},
   MRCLASS = {14J35},
  MRNUMBER = {4692375},
MRREVIEWER = {Arnaud\ Beauville},
       DOI = {10.1093/imrn/rnac160},
       URL = {https://doi.org/10.1093/imrn/rnac160},
}

@book {geometryofcubics,
    AUTHOR = {Huybrechts, Daniel},
     TITLE = {The geometry of cubic hypersurfaces},
    SERIES = {Cambridge Studies in Advanced Mathematics},
    VOLUME = {206},
 PUBLISHER = {Cambridge University Press, Cambridge},
      YEAR = {2023},
     PAGES = {xvii+441},
      %ISBN = {978-1-009-28000-6; [9781009280020]},
   MRCLASS = {14J70 (14C30 14J30 14J35)},
  MRNUMBER = {4589520},
MRREVIEWER = {Fei\ Hu},
}

@misc{luan2023irreduciblecomponentshilbertscheme,
      title={Irreducible components of Hilbert scheme of points on non-reduced curves}, 
      author={Yuze Luan},
      year={2023},
      howpublished = {Preprint, arXiv:2210.01170},
      arXiv = {arXiv:2210.01170},
      url={https://arxiv.org/abs/2210.01170}, 
}

@misc{HM1,
 author = {Daniel Huybrechts and Dominique Mattei},
 title = {The special Brauer group and twisted Picard varieties},
 year = {2023},
 howpublished = {Preprint, arXiv:2310.04032v2},
 url = {https://arxiv.org/abs/2310.04032v2},
 arXiv = {arXiv:2310.04032v2}
}

@misc{huybrechts2025periodindexproblemhyperkahlermanifolds,
 title={The period-index problem for hyperk\"ahler manifolds}, 
 author={Daniel Huybrechts},
 year = {2025},
 howpublished = {Preprint, arXiv:2411.17604v2},
 url = {https://arxiv.org/abs/2411.17604v2},
 arXiv = {arXiv:2411.17604v2}
}

@article {brauerquot,
    AUTHOR = {Biswas, Indranil and Dhillon, Ajneet and Hurtubise, Jacques},
     TITLE = {Brauer groups of {Q}uot schemes},
   JOURNAL = {Michigan Math. J.},
  FJOURNAL = {Michigan Mathematical Journal},
    VOLUME = {64},
      YEAR = {2015},
    NUMBER = {3},
     PAGES = {493--508},
      ISSN = {0026-2285,1945-2365},
   MRCLASS = {14F22 (14C05 14D23)},
  MRNUMBER = {3394256},
MRREVIEWER = {J\"{o}rg\ Jahnel},
       DOI = {10.1307/mmj/1441116655},
       URL = {https://doi.org/10.1307/mmj/1441116655},
}

@article {brauersymm,
    AUTHOR = {Iyer, Jaya N. N. and Joshua, Roy},
     TITLE = {Brauer groups of schemes associated to symmetric powers of
              smooth projective curves in arbitrary characteristics},
   JOURNAL = {J. Pure Appl. Algebra},
  FJOURNAL = {Journal of Pure and Applied Algebra},
    VOLUME = {224},
      YEAR = {2020},
    NUMBER = {3},
     PAGES = {1009--1022},
      ISSN = {0022-4049,1873-1376},
   MRCLASS = {14F22 (14C25 14D20 14D23 14F20)},
  MRNUMBER = {4009565},
MRREVIEWER = {WonTae\ Hwang},
       DOI = {10.1016/j.jpaa.2019.06.019},
       URL = {https://doi.org/10.1016/j.jpaa.2019.06.019},
}

@article {Brauerpurity,
    AUTHOR = {\v Cesnavi\v cius, K.},
     TITLE = {Purity for the {B}rauer group},
   JOURNAL = {Duke Math. J.},
  FJOURNAL = {Duke Mathematical Journal},
    VOLUME = {168},
      YEAR = {2019},
    NUMBER = {8},
     PAGES = {1461--1486},
      ISSN = {0012-7094,1547-7398},
   MRCLASS = {14F22 (14F20 14G22 16K50)},
  MRNUMBER = {3959863},
MRREVIEWER = {J\"org\ Jahnel},
       DOI = {10.1215/00127094-2018-0057},
       URL = {https://doi.org/10.1215/00127094-2018-0057},
}

@article {doublepointcurve,
    AUTHOR = {Corvaja, Pietro and Zucconi, Francesco},
     TITLE = {The surface of {G}auss double points},
   JOURNAL = {Eur. J. Math.},
  FJOURNAL = {European Journal of Mathematics},
    VOLUME = {8},
      YEAR = {2022},
    NUMBER = {3},
     PAGES = {871--892},
      ISSN = {2199-675X,2199-6768},
   MRCLASS = {14J25 (14J28 14N25)},
  MRNUMBER = {4498818},
MRREVIEWER = {Raquel\ Mallavibarrena},
       DOI = {10.1007/s40879-021-00456-x},
       URL = {https://doi.org/10.1007/s40879-021-00456-x},
}

@article {Galati-Knutsen,
    AUTHOR = {Galati, Concettina and Knutsen, Andreas Leopold},
     TITLE = {On the existence of curves with {$A_k$}-singularities on
              {$K3$} surfaces},
   JOURNAL = {Math. Res. Lett.},
  FJOURNAL = {Mathematical Research Letters},
    VOLUME = {21},
      YEAR = {2014},
    NUMBER = {5},
     PAGES = {1069--1109},
      ISSN = {1073-2780,1945-001X},
   MRCLASS = {14J28 (14B07 14H10)},
  MRNUMBER = {3294563},
MRREVIEWER = {Trygve\ Johnsen},
       DOI = {10.4310/MRL.2014.v21.n5.a8},
       URL = {https://doi.org/10.4310/MRL.2014.v21.n5.a8},
}

@article {Caporaso-Harris,
    AUTHOR = {Caporaso, Lucia and Harris, Joe},
     TITLE = {Parameter spaces for curves on surfaces and enumeration of
              rational curves},
   JOURNAL = {Compositio Math.},
  FJOURNAL = {Compositio Mathematica},
    VOLUME = {113},
      YEAR = {1998},
    NUMBER = {2},
     PAGES = {155--208},
      ISSN = {0010-437X,1570-5846},
   MRCLASS = {14N10 (14H10 14J26 14M25)},
  MRNUMBER = {1639183},
MRREVIEWER = {Ravi\ D.\ Vakil},
       DOI = {10.1023/A:1000401119940},
       URL = {https://doi.org/10.1023/A:1000401119940},
}

\end{document}